\documentclass{amsart}

\usepackage{a4wide}
\usepackage{mathtools}
\usepackage{amssymb,amsmath,amsfonts,amsthm}
\usepackage{paralist}
\usepackage{latexsym}
\usepackage{mathbbol}
\usepackage{color}

\usepackage{placeins}

\usepackage{tikz}

\newcommand{\N}{\mathbb{N}}
\newcommand{\R}{\mathbb{R}}
\newcommand{\0}{\mathbb{0}}

\newcommand{\E}{\mathcal{E}}

\newcommand{\set}[2]{\left\{#1\vphantom{#2}\right.\;\left|\;\vphantom{#1}#2\right\}}
\newcommand{\ext}[2]{{\mathsf {Ext}}^{#1}_{#2}}
\newcommand{\zvek}[1]{\begin{matrix} \textendash\textendash & \!\!\! {#1} \!\!\! & \textendash\textendash \end{matrix}}
\newcommand{\svek}[1]{\begin{matrix} | \\ {#1} \\ | \end{matrix}}
\newcommand{\zsvek}[1]{\begin{smallmatrix} \textendash\textendash &  {#1}  & \textendash\textendash \end{smallmatrix}}
\newcommand{\ssvek}[1]{\begin{smallmatrix} | \\ {#1} \\ | \end{smallmatrix}}
\newcommand{\ID}{{\mathsf {Id}}}

\makeatletter
\renewcommand*\env@matrix[1][*\c@MaxMatrixCols c]{%
  \hskip -\arraycolsep
  \let\@ifnextchar\new@ifnextchar
  \array{#1}}
\makeatother

\DeclareMathOperator{\interior}{\mathsf{int}}
\DeclareMathOperator{\conv}{\mathsf{conv}}

\DeclareMathOperator{\vol}{\mathsf{vol}}
\DeclareMathOperator{\pos}{\mathsf{pos}}
\DeclareMathOperator{\corank}{\mathsf{corank}}
\DeclareMathOperator{\Iso}{\mathsf{Iso}}
\DeclareMathOperator{\GL}{\mathsf{GL}}
\DeclareMathOperator{\Cay}{\mathsf{Cay}}

\newtheorem{theorem}{Theorem}[section]
\newtheorem{lemma}[theorem]{Lemma}
\newtheorem{proposition}[theorem]{Proposition}
\newtheorem{corollary}[theorem]{Corollary}
\theoremstyle{definition}
\newtheorem{example}[theorem]{Example}
\newtheorem{remark}[theorem]{Remark} 
\newtheorem{definition}[theorem]{Definition}

\title{Polyhedral volume ratios, Izmestiev's Colin de Verdi{\`e}re matrices and Spectral Gaps}

\author{Ioannis Ivrissimtzis, Carsten Lange, Shiping Liu, Norbert Peyerimhoff}
\date{\today}

\begin{document}

\maketitle

\begin{abstract}
   We present a relation between volumes of certain lower dimensional simplices associated to a full-dimensional primal and polar dual polytope in $\R^k$. 
   We then discuss an applications of  this relation to a geometric construction of a Colin de Verdi{\`e}re matrix by Ivan Izmestiev. In the second part of the paper, we introduce a variation of vertex transitive polytopes, translate their associated Colin de Verdi{\`e}re matrices into random walk matrices, and investigate extremality properties of the spectral gaps of these random walk matrices 
   in two concrete examples -- permutahedra of Coxeter groups and polytopes associated to the pure rotational tetrahedral group -- where maximal spectral gaps correspond to equilateral polytopes. 
\end{abstract}

\tableofcontents

\section{Introduction}
\label{sec:intro}

In the following subsections, we present our main result, a polyhedral volume ratio formula (Theorem~\ref{thm:main1}), discuss an application to a particular Colin de Verdi\`ere matrix (Corollary~\ref{cor:altIavnCvD}), and give an overview over the structure and content of this paper. 


\subsection{A volume ratio formula}

We start with some basic facts about polar duals (see, e.g., \cite[Section 3.4]{G-03} or \cite[Section 2.3]{Z-95}).
Let $C \subset \R^k$ be a (bounded or unbounded) $k$-dimensional polyhedron containing the origin $\0$ in its interior $\interior(C)$. The polar-dual of $C$, denoted by $P^* \subset \R^k$, is defined as follows:
$$ C^* = \set{p \in \R^k}{\langle p, v \rangle \le 1 \text{ for all $v \in C$}}. $$
Note that $C^*$ is not empty since it contains the origin $\0$. In fact, $C^*$ is always a polytope. In the case that $C$ is unbounded, $\0$ becomes a vertex of its polar-dual $C^*$. Moreover, if $r \subset \R^k$ is a ray generated by $v \in \R^k$ and completely contained in an unbounded polyhedron $C$, then the polar-dual $C^*$ is contained in the closed half space $H = \set{x \in \R^k}{\langle v,x \rangle \le 0}$.

\goodbreak

\begin{example} \label{example:cone-dual} $ $ 
  \begin{compactenum}[a)] 
    \item An example of a bounded polyhedron and its polar-dual is a cube and an octahedron: 
        \[
            Q = \set{(x,y,z) \in \R^3}{-2 \le x,y,z \le 2}
        \]
        and
        \[
            Q^* = \conv\left\{\pm\tfrac{1}{2}e_1, \pm\tfrac{1}{2}e_2, \pm\tfrac{1}{2}e_3\right\},
        \]
        where $e_i$ denote the standard basis vectors of $\R^3$.
    \item An example of an unbounded polyhedron and its polar-dual is a translate of an octant with apex $v_0=(-2,-2,-2)$ and a simplex (with the origin as one vertex):
        \[
            C = \set{(x,y,z) \in \R^3}{x,y,z \ge -2} = v_0 + \pos(e_1,e_2,e_3)
        \]
        and
        \[
            C^* = \conv\left\{\0,\ \tilde q_1 = -\tfrac{1}{2}e_1,\ \tilde q_2 = -\tfrac{1}{2}e_2,\ \tilde q_3 = -\tfrac{1}{2}e_3\right\}.
        \]
    In order to obtain the vertices $\tilde q_j$ of $C^*$, we first compute facet normals $q_j$ of $C$ using the vector products
    $$ q_1 = e_2 \times e_3, \quad q_2 = - e_1 \times e_3, \quad \text{and} \quad q_3 =  e_1 \times e_2, $$
    and rescale them according to polarity:
    \[
    	\tilde q_j = \frac{1}{\langle v_0, q_j \rangle} q_j. 
    \]
    In the general case $C = v_0 + \pos(\varepsilon_1,\dots,\varepsilon_k) \subset \R^k$, facet normals $q_j$ and $\tilde q_j$ are computed similarly via the generalized multilinear vector product
    \begin{equation} \label{eq:qj} 
    	q_j = (-1)^{j-1} \varepsilon_1 \times \cdots \times \widehat \varepsilon_j \times \cdots \times \varepsilon_k
		\qquad\text{and}\qquad
		\tilde q_j = \frac{1}{\langle v_0, q_j \rangle} q_j,
    \end{equation}
    where $\widehat \epsilon_j$ means that this factor is omitted.
  \end{compactenum}
\end{example}

Our main theorem in this section provides a ratio between volumes of two polytopes $S_I$ and~$\Delta_J$, which are associated to a cone $C$ and its polar-dual, respectively.

\begin{theorem} \label{thm:main1}
  Let $C \subset \R^k$ be the following $k$-dimensional simplicial cone with apex $v_0 \in \R^k$ and extremal rays generated by $\varepsilon_1,\dots,\varepsilon_k$:
  \begin{eqnarray*} 
  	C &=& v_0 + \pos(\varepsilon_1,\dots,\varepsilon_k) \\
  		&=& \set{v \in \R^k}{\langle v,\tilde q_j \rangle \leq \langle v_0,\tilde q_j \rangle \text{ for $j \in [k]$} }
  \end{eqnarray*}
  with the outer facet normals $\tilde q_j$ in \eqref{eq:qj}. Moreover, we assume $\0 \in \interior(C)$. Let $C^* \subset \R^k$ be the polar-dual of $C$:
  \[
    C^* = \conv\left(\{\0\}\cup\set{\tilde q_j}{j\in[k]}\right).
  \]  
  For $0 \le m < k$ consider $I \in {{[k]} \choose m}$
  and its complement $J := [k] \setminus I$ with $\ell := |J|$. Define the $(m+1)$-simplex $S_I \subset C$ and the 
  $(\ell-1)$-simplex $\Delta_J \subset C^*$ as:
  $$ S_I := \conv\left( \{ \0, v_0 \} \cup \{ v_0 + \varepsilon_i \mid i \in I \}\right)
  \qquad\text{and}\qquad
  \Delta_J := \conv( \{ \tilde q_j \mid j \in J \}).
  $$
  Then we have
  \begin{equation} \label{eq:main-formula}
  \frac{\vol_{\ell-1}(\Delta_J)}{\vol_{m+1}(S_I)} = \frac{(m+1)!}{(l-1)!} \cdot \frac{\vol_k(\E)^{\ell-1}}{\prod_{j \in J} |\langle v_0, q_j \rangle|},  
  \end{equation}
  where $\E$ is the $k$-dimensional parallel-epiped spanned by $\varepsilon_1,\dots,\varepsilon_k$.
\end{theorem}

Let us illustrate this result with the following example.

\begin{example}
  Let $$C = v_0 + \pos(\varepsilon_1 = e_1, \varepsilon_2 = e_2,\varepsilon_3 = e_3) $$ 
  be the cone introduced in part b) of Example \ref{example:cone-dual}. Choosing $I = \{1,2\}$ and $J = \{3\}$ in Theorem \ref{thm:main1} leads to the $0$-simplex
  $\Delta_J = \{ \tilde q_3 \}$ and $S_I$ is a $3$-dimensional pyramid with base $\conv(v_0,v_0+e_1,v_0+e_2)$ and apex $\0$, see Figure ..., and the volume ratio reads as
  $$ \frac{\vol_{0}(\Delta_J)}{\vol_{3}(S_I)} = \frac{3!}{0!} \cdot \frac{\vol_3(\E)^{0}}{|\langle v_0, q_3 \rangle|} = 3. $$
  
  Choosing $I = \{1\}$ and $J = \{2,3\}$ in Theorem \ref{thm:main1} leads to the $1$-simplex $\Delta_J = \conv(\tilde q_2,\tilde q_3)$ (a line segment) and $S_I$ is the $2$-simplex (a triangle) with base $\conv(v_0,v_0+e_1)$ and apex $\0$, see Figure ..., and the volume ratio reads as
  $$ \frac{\vol_{1}(\Delta_J)}{\vol_{2}(S_I)} = \frac{2!}{1!} \cdot \frac{\vol_3(\E)^{1}}{|\langle v_0, q_2 \rangle| \cdot |\langle v_0, q_3 \rangle|} = \frac{1}{2}. $$
  Let us briefly discuss the modification appearing if we rescale the generating vectors by $\lambda > 0$, that is $\varepsilon_j = \lambda e_j$. 
\end{example}

\begin{remark}
  Uniform scaling the generating vectors $\varepsilon_j$ by a common factor $\lambda > 0$ 
  yields rescaled vectors $q_j^\lambda = \lambda^{k-1} q_j$, $\tilde q_j^\lambda = \tilde q_j$ and volumes $\vol_{l-1}(\Delta_J^\lambda) = \vol_{l-1}(\Delta_j)$, $\vol_{m+1}(S_I^\lambda) = \lambda^m \vol_{m+1}(S_I)$ and $\vol_k(\E^\lambda) = \lambda^k \vol_k(\E)$. Thus both sides of~\eqref{eq:main-formula} rescale by the same factor $\lambda^{-m}$. 
\end{remark}  
  
\begin{remark} \label{rem:pol2con}
Theorem~\ref{thm:main1} can be applied to each vertex of a simple $k$-polytope $Q \subset \R^k$ containing the origin as follows. (Recall that a $k$-polytope $Q \subset \R^k$ is called \emph{simple} if each of each vertices is adjacent to exactly $k$-edges.) Choose a vertex $v_0 \in Q$ and consider its neighbours $v_1,\dots,v_k$ and its edge vectors $\varepsilon_j = v_j - v_0$. Then the $r$-faces of $C=v_0 + \pos(\varepsilon_1,\dots,\varepsilon_k)$ are in bijection to $r$-faces of $Q$ incident to $v_0$ and these faces of $Q$ are contained in the corresponding facets of $C$ for all $r\in[k]$. Now we choose $0 \le m < k$, $I \in {[k] \choose m}$ and $J = [k] \setminus I$. Using the notation in Theorem \ref{thm:main1}, $S_I$ can be viewed as a pyramid with apex $\0$ and base $F_I$ contained in the $m$-face of $Q$ incident to the vertices $v_0$ and $v_i$ with $i \in I$. The polar-dual $Q^*$ is simplicial and has a facet corresponding to $v_0 \in Q$. Then $\Delta_J$ is an $(|J|-1)$-face of $Q^*$ with vertices $\tilde q_j$, $j \in J$ and the volume ratio between $\Delta_J$ and $S_I$ is given in \eqref{eq:main-formula}.

\end{remark}

The proof of Theorem \ref{thm:main1} above is quite involved and is presented in Sections \ref{sec:proof_main_thm} and \ref{sec:three_fundamental_statements}. Readers who are mainly interested in a special Colin de Verdi{\'e}re Matrix application and spectral gap considerations can skip these two sections and turn directly to 
Section \ref{sec:vertextranspoly} on vertex transitive polytopes.

\subsection{An application to Izmestiev's Colin de Verdi{\`e}re Matrix}


Yves Colin de Verdi{\`e}re introduced a spectral graph invariant~$\mu$ which is the maximum corank of a family of so-called Colin de Verdi{\`e}re Matrices 
whose non-diagonal entries resemble the adjacency matrix in a weighed sense. This invariant has some remarkable properties: it is monotone under taking minors 
and its values characterize outer-planarity ($\mu \le 2$), planarity ($\mu \le 3$) and linkless embeddability ($\mu \le 4$) (see \cite{CdV-90,LS-98,vdHLS-99}). Generalizing earlier work by 
Lov{\'a}sz \cite{Lov-01} to higher dimensions $k > 3$, Izmestiev presented an explicit geometric construction of a Colin de Verdi{\`e}re Matrix of corank $k$ for 
the vertex-edge graph of any $k$-polytope in $\R^k$. We will see that, for simple polytopes, every non-zero off-diagonal entry of Izmestiev's matrix is a certain volume ratio which coincides with the volume ratio of Theorem~\ref{thm:main1} in the case $m=1$. We now define Colin de Verdi{\`e}re Matrices and state Izmestiev's result.

\begin{definition}[see {\cite{CdV-90,CdV-98,CdV-04}}]
  Let $G$ be a finite graph with vertex set $[n]$. A symmetric matrix $M = (M_{st}) \in \R^{n \times n}$ is called a \emph{Colin de Verdi{\`e}re Matrix} 
  for~$G$, if the following three properties are satisfied
  \begin{compactenum}[(M1)]
	\item 
		\[
		    M_{st} \begin{cases} 
						< 0 	& \text{if $st$ is an edge of $G$}, \\ 
						= 0 	& \text{if $st$ is not an edge of $G$ and $s \neq t$,}\\
						\in\R 	& \text{if $s=t$.}
					\end{cases}
		\]
	\item $M$ has exactly one negative eigenvalue and this eigenvalue is simple.
	\item If $X \in \R^{n\times n}$ is symmetric and $MX = \0$ and $X_{st}=0$ whenever $s=t$ or $st$ is an edge of $G$, then $X=\0$. 
  \end{compactenum}
  Let $\mathcal{M}_G$ be the set of all Colin de Verdi\`ere matrices of $G$. 
  The \emph{Colin de Verdi{\`e}re Number} $\mu(G)$ is defined as
  $$ \mu(G) = \max_{M \in \mathcal{M}_G} \corank(M). $$
  A Colin de Verdi{\`ere} Matrix of maximal corank is called \emph{optimal}.
\end{definition}

\begin{theorem}[see {\cite[Theorem 2.4]{Izm-10}} and its proof] \label{thm:ivan}
 	Let $Q \subset \R^k$ be a $k$-polytope with vertices $w_1,\dots,w_n$ that contains the origin $\0$ in its interior, $G$ be its vertex-edge graph 
	(or $1$-skeleton) and $P \subset \R^k$ be the polar-dual of $Q$. The vertex of $G$ corresponding to $w_s \in Q$ is denoted by $s \in [n]$, and the vertex set of $G$ is therefore $[n]$. For a given vertex $w_s$ of and a neighbour $w_t$ in $Q$, let $F_{st}$ be the ridge of $P$ obtained by the intersection of the facets of $P$ with outer 
	normals $w_s$ and $w_t$ and $\theta_{st}$ be the associated dihedral angle. 
  
 	 A Colin de Verdi{\`e}re Matrix $M \in \R^{n \times n}$ of corank $k$ for $G$ is then given by
	 \begin{equation} \label{eq:Mst}
     	M_{st} = 
			\begin{cases}
				- \frac{\vol_{k-2}(F_{st})} {\Vert w_s \Vert \cdot \Vert w_t \Vert \sin \theta_{st}} & \text{if $st$ is an edge of $G$}, \\ 
				0 																					 & \text{if $st$ is not an edge of $G$ and $s \neq t$,} 
			\end{cases}
	\end{equation}
	and
	\begin{equation} \label{eq:MIvan}
		M_{ss}w_s = - \sum_{t \neq s} M_{st} w_t.
	\end{equation}
    We denote this Colin de Verdi{\`e}re Matrix by $M(Q)$.
    In particular, we have for the Colin de Verdi{\`e}re Number
    $$ \mu(G) \ge k. $$
\end{theorem}

\begin{remark}
Let us provide more details about Izmestiev's matrix $M$. Inspired by Lov{\'a}sz construction in the $3$-dimensional case \cite{Lov-01}, Izmestiev defined for all $s,t \in [n]$
\begin{equation} \label{eq:Mstorig} 
M_{st} = - \frac{\partial^2 \vol(P(x))}{\partial x_s \partial x_t}\vert_{x = (1,\dots,1)}, 
\end{equation}
where 
$$ P(x) = \set{p \in \R^k}{\langle p, w_s \rangle \le x_s \text{ for all $s \in [n]$}} $$
and $P(1,\dots,1) = Q^*$. Note that, in the case that $s \neq t$ and $w_s$ and $w_t$ are not neighbouring vertices in $Q$, the expression on the right hand side of \eqref{eq:Mstorig} vanishes. In the case that $st$ is an edge, he used the identity
$$ \frac{\partial^2 \vol(P(x))}{\partial x_s \partial x_t}\vert_{x = (1,\dots,1)} = \frac{\vol_{k-2}(F_{st})} {\Vert w_s \Vert \cdot \Vert w_t \Vert \sin \theta_{st}}. $$
The formula \eqref{eq:MIvan} for $M_{ss}$ is a consequence of
\begin{equation} \label{eq:gaussdiv} 
\sum_{t \in [n]} {\rm{vol}}(F_t(x)) \frac{w_t}{\Vert w_t \Vert} = 0 
\end{equation}
via differentiation with respect to $x_s$,
where $F_t(x)$ is the facet of $P(x)$ corresponding to the facet normal $w_t$. The resulting identity (see \cite[formula (5)]{Izm-10})
$$ \frac{\partial^2 {\rm{vol}}(P(x))}{\partial x_s^2} w_s + \sum_{t \neq s} \frac{\partial^2 {\rm{vol}}(P(x)}{\partial x_s \partial x_t} w_t = 0 $$
then implies the important fact that the right hand side of \eqref{eq:MIvan} is parallel to $w_s$ with scaling factor $M_{ss}$.\footnote{The scaling factor $M_{ss}$ can be zero, for example, in the case of the regular octahedron.} Note that
the identity \eqref{eq:gaussdiv} is a discrete version of the fact that the integral of the outer unit normal over a closed orientable hypersurface is zero, which is a direct consequence of Gauss Divergence Theorem. 
\end{remark}

A straightforward but useful consequence of Theorem \ref{thm:ivan} is Corollary \ref{cor:specrep} below, which is related to the inverse construction of polytopes from Colin de Verdi{\`e}re matrices of general graphs discussed in \cite[Subsection 3.1]{Izm-10}.  While this inverse construction is not always possible for general graphs, Corollary \ref{cor:specrep} is a positive answer to the question whether a given $k$-polytope $Q$ can be spectrally reconstructed -- at least combinatorially -- from the Colin de Verdi{\`e}re matrix $M$ given in Theorem \ref{thm:ivan}. In the case of the adjacency matrix, the spectral construction of polytopes from eigenspaces lead to the notion of \emph{eigenpolytopes} by Godsil (see \cite{G-78} and Winter \cite{W-20} for a survey and results in the special case of edge transitive polytopes).

\begin{corollary} \label{cor:specrep}
    Let $Q \subset \mathbb{R}^k$ be a $k$-polytope with $n$ vertices $w_1,\dots,w_n$ and
    \begin{equation} \label{eq:wphi} 
    \begin{psmallmatrix} \zsvek{w_1} \\[-1mm] \vdots \\ \zsvek{w_n} \end{psmallmatrix} = \begin{psmallmatrix} \ssvek{\phi_1} & \cdots & \ssvek{\phi_k} \end{psmallmatrix}. \end{equation}
    Then the vectors $\phi_j \in \mathbb{R}^n$ are a basis of the kernel of the Colin de Verdi{\`e}re Matrix $M$ associated to $Q$ in Theorem \ref{thm:ivan}.
\end{corollary}

\begin{proof}
    Recall from \eqref{eq:MIvan} that we have for all fixed $s \in [n]$,
    $$ \sum_{t} M_{st} w_t = 0, $$
    which gives
    $$ M \begin{psmallmatrix} \ssvek{\phi_1} & \cdots & \ssvek{\phi_k} \end{psmallmatrix} = M \begin{psmallmatrix} \zsvek{w_1} \\[-1mm] \vdots \\ \zsvek{w_n} \end{psmallmatrix} = \begin{psmallmatrix} \zsvek{\mathbb 0_k^\top} \\[-1mm] \vdots \\ \zsvek{\mathbb 0_k^\top}  \end{psmallmatrix} = \begin{psmallmatrix} \ssvek{\mathbb 0_n} & \cdots & \ssvek{\mathbb 0_n} \end{psmallmatrix}. $$ 
    Therefore, $\phi_1,\dots,\phi_k$ lie in the kernel of $M$. Since
    $Q$ is a $k$-polytope, its vertices $w_1,\dots,w_n$ span $\mathbb{R}^k$, and since $M$ has of corank $k$, the vectors $\phi_1,\dots,\phi_k$ form as basis of this kernel.
\end{proof}    


In the case of a \emph{simple} $k$-polytope $Q \subset \mathbb{R}^k$, we can use Theorem \ref{thm:main1} to obtain an alternative description of the entries of the Colin de Verdi{\`e}re Matrix $M$. Assume that $Q$ is simple and that there is an edge between its vertices $w_s$ and $w_t$. To connect 
to Remark~\ref{rem:pol2con}, recall that $v_0$ is the apex of the cone $C$ and the neighbors of $v_0$ are denoted by $v_1,\dots,v_k$. Now we identify $w_s$ 
with $v_0$ and $w_t$ is $v_r$ for some $r\in[k]$. We also have $m=1$, $I=\{r\}$ and $J=[k]\setminus I$. The denominator 
$\Vert w_s \Vert \cdot \Vert w_t \Vert \sin \theta_{st}$ 
of $M_{st}$ is the volume of the $2$-dimensional parallel-epiped spanned by the vectors $v_0,v_r \in \R^k$. Thus
\[
	 \Vert w_s \Vert \cdot \Vert w_t \Vert \sin \theta_{st} = 2 \vol_2(S_I). 
\]
The numerator of $M_{st}$ is the volume of the ridge $F_{st}$ which is the intersection of the facets polar-dual to the vertices $v_0$ and $v_r$ and 
therefore the convex hull of $\tilde q_j$, where $j\in J$. 
Thus $F_{st}=\conv\set{\tilde q_j}{j\in J}=\Delta_J$ and
\[
	\vol_{k-2}(F_{st}) = \vol_{k-2}(\Delta_J).
\]

With Theorem \ref{thm:main1} we obtain therefore the following new interpretation of the non-zero and off-diagonal entries of M for a simple polytope $Q$.

\begin{corollary} \label{cor:altIavnCvD}
	Let $Q \subset \R^k$ be a $k$-polytope with vertices $w_1,\dots,w_n$ that contains the origin in its interior and $M$ the Colin de Verdi{\`e}re Matrix of 
	Theorem~\ref{thm:ivan}. If $Q$ is \emph{simple} and $w_s, w_t$ is a pair of vertices that form an edge of $Q$ then 
	\begin{equation} \label{eq:Mstalternative}
     	M_{st} = 
				- \frac{1}{(k-2)!} 
					\cdot 
					\frac{\vol_k(\E)^{k-2}}{\prod_{j \in [k]\setminus \{r\}} |\langle w_s, q_j \rangle|},
	\end{equation}
	where $q_{j} = (-1)^{j-1} \varepsilon_1 \times \cdots \times \widehat{\varepsilon}_j \times \cdots \times \varepsilon_k$ with $v_1, \dots, v_k$ be the neighbors of $v_0:=w_s$ 
	in $Q$ with $v_r=w_t$ and $\varepsilon_j = v_j - v_0$ for $j \in [k]$ and $\E$ is the parallel-epiped spanned by $\varepsilon_1,\dots,\varepsilon_k \in \R^k$.
\end{corollary}

We would like to note that the right hand side of \eqref{eq:Mstalternative} may at times be easier to compute than the original definition \eqref{eq:Mst} of $M_{st}$ since it only involves the vertices of the polytope $Q$ and vector products of certain of its differences and not geometric quantities involving both the polytope $Q$ and its polar dual $P$. We will see this in Subsection \ref{subsec:coxperm} when we study permutahedra of Coxeter groups (Corollary \ref{cor:CoxMdescr}). 

\subsection{Spectral graps of vertex transitive polytopes}

Permutahedra of Coxeter groups are special examples of vertex transitive polytopes, which are the main topic of Section \ref{sec:vertextranspoly}. 
Instead of just considering a single vertex transitive polytope in $\R^k$ with respect to finite subgroup $\Gamma \subset O(k)$, we consider a \emph{variation} of this polytope, where each polytope in the variation is obtained as the convex hull of the orbit $\Gamma w$ of a point $w$ in the sphere $S^{k-1}(r)$ of radius $r>0$ centered at the origin of $\R^k$. We denote this polytope by $Q_w$. The points $w \in S^{k-1}(r)$, for which there is a \emph{simply transitive} $\Gamma$-action on the vertices of $Q_w$, form a subset of $S^{k-1}(r)$ of full measure, and we refer to this set as the \emph{regular subset} $S^{k-1}_{\rm{reg}}(r)$. In fact, the sphere $S^{k-1}(r)$ and its regular subset $S^{k-1}_{\rm{reg}}(r)$ can be partitioned into subsets in which the associated polytopes have specific combinatorial types. Next we consider the Colin de Verdi\`ere Matrices introduced in Theorem \ref{sec:vertextranspoly} for the polytopes in this variation. These Colin de Verdi{\`e}re Matrices $M(Q_w)$ allow a natural transformation into doubly stochastic matrices, which can be viewed as random walk matrices on the vertex-edge graphs of the polytopes $Q_w$. We denote this random walk matrices by $T(Q_w)$ (see \eqref{eq:TtauM} with the choice $\tau=\gamma$). In Section \ref{sec:vertextranspoly}, we study the second largest    
eigenvalue $\lambda_1$ of these random walk matrices under the variation by way of specific examples: we discuss permutahedra of Coxeter groups in Subsection \ref{subsec:coxperm} and we investigate polytopes associated to the pure rotational tetrahedral 
group in Subsection \ref{subsec:rottet}. In these examples, we find a close relationship between equilaterality of the polytopes $Q_w$ in the variation and minimality properties of $\lambda_1$. Since the largest eigenvalue $\lambda_0$ of a stochastic matrices is always equal to $1$, minimality of $\lambda_1$ corresponds to maximality of the spectral gap $\lambda_0-\lambda_1$, which measures the convergence rate of the random walk, described by $T(Q_w)$, on the vertex-edge graph of $Q_w$ to the equidistribution. These observations lead to a discussion and questions for further research in the final Subsection \ref{subsec:discusquestions}.


\section{Proof of our main theorem}
\label{sec:proof_main_thm}

This section is dedicated to the proof of our main result, 
the volume formula in Theorem~\ref{thm:main1}. Let us first outline the strategy of proof. In Lemma~\ref{lem:volume_S_I}, a formula for the volume of 
the $(|I|+1)$-simplex $S_I=\mathsf{conf}\{\0, v_0\}\cup\set{v_0+\varepsilon_i}{i\in I}$ is derived, while a formula for the volume 
of the $(|J|-1)$-dimensional simplex $\Delta_J := \mathsf{conv} \set{ \widetilde q_j }{j\in J}\subseteq \Delta_{[k]}$ is obtained 
in Lemma~\ref{lem:volume_Delta_J}. Finally, the statement in Theorem~\ref{thm:main1} about the volume ratio $\frac{\mathsf{vol}_{|J|-1}(\Delta_J)}{\mathsf{vol}_{|I|+1}(S_I)}$ follows by combining these two results. In the proofs of this section, we use some fundamental lemmas of own interest, which are presented in the following Section \ref{sec:three_fundamental_statements}.

We start with some notation that we use throughout the paper. Recall that the adjunct matrix~$\mathsf{adj}\, A$ of 
$A\in \R^{r\times r}$ is the transpose of its cofactor matrix, that is,
\[
	\big( \mathsf{adj}\, A \big)_{s,t}
		=
		(-1)^{s+t}\det(A^{t,s})
	\qquad\text{for all $s,t\in[r]$.}
\]
where $A^{t,s}$ is obtained from $A$ by deletion of row $t$ and column $s$. A standard result of linear algebra is 
\[
	A\cdot \mathsf{adj}\, A = \mathsf{adj}\, A\cdot A = (\det A)\ID_r.
\]

The restriction of a matrix is obtained by deleting certain rows and columns simultaneously and the extension is obtained by inserting certain rows and columns with all entries equal to zero. Both processes are described in the following definition.

\begin{definition} $ $\\
	Let $r,s\in\N$ and $I=\{\ell_1<\ell_2<\ldots<\ell_s\}\in\binom{[r]}{s}$.
	\begin{compactenum}[a)]
		\item The restriction $[A]_{I}=(\widetilde a_{ij})_{i,j\in[s]}$ of $A=(a_{ij})_{i,j\in[r]}$ is defined by 
			\[
				\widetilde a_{ij} := a_{\ell_i\ell_j}\qquad \text{for all $i,j\in[s]$.}
			\]
		\item The extension $\ext{[r]}{I} (B)=(\widetilde b_{ij})_{i,j\in[r]}$ of $B=(b_{ij})_{i,j\in[s]}$ is defined by 
			\[	
				\widetilde b_{ij}
					:= \begin{cases} 
						b_{\varphi(i)\varphi(j)} 	& i,j\in I\\ 
						0 							& i\in [r]\setminus I\text{ or }j\in[r]\setminus I,
					  \end{cases}
			\]
			where $\varphi:I\rightarrow[s]$ with $\varphi(\ell_k):=k$.
	\end{compactenum} 
\end{definition}

Now we can start with our proof of Theorem \ref{thm:main1}. Our first lemma reads as follows:

\begin{lemma}$ $\\ \label{lem:volume_S_I}\noindent
	Let $C_{v_0}(\varepsilon_1,\ldots,\varepsilon_k) = v_0 + \pos(\varepsilon_1,\dots,\varepsilon_k) \subset \R^k$ be a $k$-dimensional simplicial cone which contains~$\mathbb 0$ in its interior and $S_I:=\conv(\{\mathbb 0, v_0\}\cup\set{v_0+\varepsilon_i}{i\in I})$ for $I\in\binom{[k]}{m}$ and $m\in [k]$ with $m < k$.
	Then
	\begin{align*}
		\big(\mathsf{vol}_{m+1}(S_I)\big)^2 
			&= \left(\frac{1}{(m+1)!}\right)^2
				\beta^\top
				\cdot
				\Big(
					\big(\det [E]_{I}\big) E^{-1}
					-  \mathsf{Ext}^{[k]}_I\big(\mathsf{adj}\,[E]_I\big)
				\Big)
				\cdot 
				\beta	
	\end{align*}
	where
	\[
		E		:= \Big( 
						\langle \varepsilon_{r},\varepsilon_{s} \rangle 
			   	   \Big) 
				   \in \R^{k\times k}
		\qquad\text{and}\qquad
		\beta:=\begin{psmallmatrix} \langle v_0,\varepsilon_1 \rangle\\[-1mm] \vdots \\ \langle v_0,\varepsilon_k \rangle \end{psmallmatrix}.
	\]
\end{lemma}

\begin{proof}
	Without loss of generality, we assume that $I=[m]$ and write $[m]_0$ for $\{0\}\cup[m]$. Moreover, let
	$v_i=v_0 + \varepsilon_i$ for all $i\in I$. Setting $\varepsilon_0=\mathbb 0$ we obtain, by using Gram's determinant formula for the volume of a parallel-epiped and then subtracting the first column from the later columns and then subtracting the first row from the later rows,
	\begin{align*}
		\Big((m+1)!\ \mathsf{vol}_{m+1}(S_I)\Big)^2 
			&= \det \Big(\langle v_i-\mathbb 0,v_j-\mathbb 0 \rangle \Big)_{i,j\in [m]_0}\\
			&= \det \Big(\langle v_0 + \varepsilon_i,v_0 + \varepsilon_j \rangle \Big)_{i,j\in [m]_0}\\
			&= \det\left(
			        \begin{pmatrix}  
							\zvek{\quad\, v_0^\top\quad\, } \\ 
							\zvek{v_0^\top + \varepsilon_1^\top} \\ 
							\vdots \\ 
							\zvek{v_0^\top + \varepsilon_m^\top} 
						\end{pmatrix}
						\cdot
						\begin{pmatrix}  
							\svek{v_0} & \svek{v_0+\varepsilon_1} & \cdots & \svek{v_0+\varepsilon_m} 
						\end{pmatrix}
					\right)\\
			&= \det
					\left(
                    \begin{pmatrix} 
						\zvek{ v_0^\top } \\ 
						\zvek{\varepsilon_1^\top} \\ 
						\vdots \\ 
						\zvek{\varepsilon_m^\top} 
					\end{pmatrix}
				\cdot
					\begin{pmatrix}  
						\svek{v_0} & \svek{\varepsilon_1} & \cdots & \svek{\varepsilon_m} 
					\end{pmatrix}
                    \right) \\
			 &= \det 
			    \left(\,\begin{matrix}[c|ccc] 
			   			\langle v_0,v_0 \rangle				& \langle v_0,\varepsilon_1 \rangle				& \cdots	& \langle v_0,\varepsilon_m \rangle				  \\ \hline 
			   			\langle \varepsilon_1,v_0 \rangle	& \langle \varepsilon_1,\varepsilon_1 \rangle 	& \cdots	& \langle \varepsilon_1,\varepsilon_m \rangle		\\
			   			\vdots								& \vdots										&			& \vdots										  \\
			   			\langle \varepsilon_m,v_0 \rangle	& \langle \varepsilon_m,\varepsilon_1 \rangle	& \cdots	& \langle \varepsilon_m,\varepsilon_m \rangle
			    \end{matrix}\, \right).
	\end{align*}
	Observe that $[\beta]_{[m]}=\begin{psmallmatrix} \langle v_0,\varepsilon_i \rangle\\ \vdots \\ \langle v_0,\varepsilon_m \rangle \end{psmallmatrix}$ and
	apply Laplace's Theorem to the first column and to the first row: 
	\[
		   \det 
		   \left(\,\begin{matrix}[c|ccc] 
					\langle v_0,v_0 \rangle		& ([\beta]_{[m]})^\top	\\ \hline 
					[\beta]_{[m]}				& [E]_{[m]}		\\				
		   \end{matrix}\, \right)
			 = \langle v_0,v_0 \rangle \det [E]_{[m]}
					+
					\sum_{i\in [m]}
						(-1)^i
						\sum_{j\in [m]}
						(-1)^{j-1}
						\beta_i \beta_j 
						\det (([E]_{[m]})^{j,i}).
	\]
	If we assume $v_0=\sum_{i\in[k]}c_i\varepsilon_i$ 
	then $E^{-1}\beta=\begin{psmallmatrix} c_1 \\ \vdots \\ c_k \end{psmallmatrix}$. 
	This implies $\langle v_0,v_0 \rangle = \beta^\top E^{-1} \beta$. Therefore
	\[
	   \det 
	   \left(\,\begin{matrix}[c|ccc] 
				\langle v_0,v_0 \rangle		& ([\beta]_{[m]})^\top	\\ \hline 
				[\beta]_{[m]}				& [E]_{[m]}		\\				
	   \end{matrix}\, \right)
		 = \det [E]_{[m]}\big(\beta^\top E^{-1} \beta\big)
					-
					[\beta]_{[m]}^\top
					\cdot
					\mathsf{adj} ([E]_{[m]})
					\cdot
					[\beta]_{[m]}
	\]
	and the claim follows.
\end{proof}

Our second lemma is concerned with the volume of the simplex $\Delta_J$.

\begin{lemma}$ $\\ \label{lem:volume_Delta_J}\noindent
    Let $C_{v_0}(\varepsilon_1,\ldots,\varepsilon_k) = v_0 + \pos(\varepsilon_1,\dots,\varepsilon_k) \subset \R^k$ be a $k$-dimensional simplicial cone which contains~$\mathbb 0$ in its interior with facet normals $q_j := (-1)^{j-1} \varepsilon_1 \times \ldots \times \widehat{\varepsilon_j} \times \ldots \times \varepsilon_k$ for all $j\in [k]$. Let $\Delta_J=\conv\set{\widetilde q_j=\frac{1}{\langle v_0,q_j\rangle}q_j}{j\in J}$ for $J\in\binom{[k]}{\ell}$ and $\ell\in[k]$. 
	Then 
	\[
		\big(\mathsf{vol}_{\ell-1}(\Delta_J)\big)^2 
			= \left(\frac{1}{(\ell-1)!\prod_{j\in J}\langle v_0,q_j\rangle}\right)^2
				v_q^\top \cdot \mathsf{Ext}^{[k]}_J\big( {\mathsf{adj}}\, [Q]_J\big)  \cdot v_q
	\]
	where
	\[
		Q		:= \Big( 
						\langle q_r,q_s \rangle 
			   	   \Big) 
				   \in \R^{k\times k}
		\qquad\text{and}\qquad
		v_q:=\begin{psmallmatrix} \langle v_0,q_1 \rangle\\[-1mm] \vdots \\ \langle v_0,q_k \rangle \end{psmallmatrix}.
	\]
\end{lemma}

\begin{proof}
	Without loss of generality, we may assume $J=[\ell]$. Then Gram's determinant formula for the volume of a parallel-epiped implies
	\[
		\Big((\ell-1)!\ \mathsf{vol}_{\ell-1}(\Delta_J)\Big)^2
			=   \det\Big( 
						\langle \widetilde q_1-\widetilde q_i, \widetilde q_1-\widetilde q_j\rangle 
					\Big)_{i,j\in [\ell]\setminus\{1\}}. 
	\]
	Because of
	\[
		\Big( \langle \widetilde q_1-\widetilde q_i, \widetilde q_1-\widetilde q_j\rangle  \Big)_{i,j\in [\ell]\setminus\{1\}}
		= 
		\begin{pmatrix}[c|c] \ssvek{\mathbb 1} & -\ID_{\ell-1}\end{pmatrix}
			\cdot
			\Big( \langle \widetilde q_i, \widetilde q_j \rangle \Big)_{i,j\in [\ell]}
			\cdot
			\begin{pmatrix}[c|c] \ssvek{\mathbb 1} & -\ID_{\ell-1}\end{pmatrix}^\top		
	\]
	the Determinant Lemma~\ref{lem:determinant_lemma} in the next section implies
	\[
		\Big((\ell-1)!\ \mathsf{vol}_{\ell-1}(\Delta_J)\Big)^2
			=   \mathbb 1_\ell^\top 
				\cdot 
				\mathsf{adj}\Big( (\langle \widetilde q_i, \widetilde q_j \rangle)_{i,j\in [\ell]} \Big)
				\cdot \mathbb 1_\ell.
	\]
	Since
	\[
		\Big( \langle \widetilde q_i, \widetilde q_j \rangle \Big)_{i,j\in [\ell]}
			= \Big( 
				\frac{1}{\langle v_0,q_i\rangle \cdot \langle v_0,q_j\rangle}
				\langle q_i, q_j \rangle
			  \Big)_{i,j\in [\ell]}
			= \Big[\Big(D_{v_q}\Big)^{-1} \cdot Q \cdot \Big(D_{v_q}\Big)^{-1}\Big]_{[\ell]},
	\]
	where $D_{v_q}$ denotes the diagonal matrix with entries $\{v_q\}_{q \in [k]}$, an application of the Scaling Lemma~\ref{lem:adjunct_lemma} implies the claim:
	\begin{align*}
		\Big((\ell-1)!\ \mathsf{vol}_{\ell-1}(\Delta_J)\Big)^2
			&=      \sum_{i,j\in [\ell]} 
						\left(
							\mathsf{adj}\left[\big(D_{v_q}\big)^{-1} \cdot Q \cdot \big(D_{v_q}\big)^{-1} \right]_{[\ell]}
						\right)_{i,j}\\
			&= \left(\frac{1}{\prod_{j\in [\ell]}\langle v_0,q_j\rangle}\right)^2
					\sum_{i,j\in [\ell]} 
						\langle v_0,q_i\rangle \cdot
						\langle v_0,q_j\rangle \cdot
						\left(
							\mathsf{adj}\, [Q]_{[\ell]} 
						\right)_{i,j}\\
	\end{align*}
\end{proof}

\begin{proof}[Proof of Theorem \ref{thm:main1}]
	From Lemma~\ref{lem:volume_S_I} we have
	\[
		\Big((m+1)!\,\,\mathsf{vol}_{m+1}(S_I)\Big)^2
			= \beta^\top
				\cdot
				\Big(
					\big(\det [E]_{I}\big) E^{-1}
					-  \mathsf{Ext}^{[k]}_I\big(\mathsf{adj}\,[E]_I\big)
				\Big)
				\cdot 
				\beta
	\]
	Since $\det(E) \neq 0$, using first the Extension Lemma~\ref{prop:extension_lemma} (with $M:=E$, $r:=k$, $a:=|I|$ and $r-a:=\ell$) and then the
	Scaling Lemma~\ref{lem:adjunct_lemma} 
    (with $A := [E^{-1}]_J, r:=\ell$, $D_\lambda := D_\mu := \sqrt{\det E} \cdot {\rm{Id}}_\ell$), we obtain
	\begin{align*}
		\big(\det [E]_{I}\big) E^{-1}  -  \mathsf{Ext}^{[k]}_I\big(\mathsf{adj}\, [E]_I\big)
			 &= \frac{1}{(\det E)^{\ell-1}}
			 		\Big( (\det E)\cdot E^{-1}\Big)
					\Big(\mathsf{Ext}^{[k]}_J\big( {\mathsf{adj}}\, [(\det E)\cdot E^{-1}]_J\big)\Big)
					E^{-1}\\
			 &= (\det E)\cdot E^{-1}\cdot \mathsf{Ext}^{[k]}_J\big( {\mathsf{adj}}\, [E^{-1}]_J\big) \cdot E^{-1}.
	\end{align*}
	If $W=\det(\varepsilon_1,\ldots,\varepsilon_k)$ then the definition of $q_i$ implies $\langle q_i,\varepsilon_j\rangle = W\delta_{ij}$ and thus
	\[
		\begin{psmallmatrix} \zsvek{q_1} \\[-1mm] \vdots \\ \zsvek{q_k} \end{psmallmatrix}
			= W
			  \begin{psmallmatrix} \ssvek{\varepsilon_1} & \cdots & \ssvek{\varepsilon_k} \end{psmallmatrix}^{-1}.
	\]
	As $Q = W^2 E^{-1}$,
    \begin{equation} \label{eq:vqE-1beta}
    v_q = \begin{psmallmatrix} \zsvek{q_1} \\[-1mm] \vdots \\ \zsvek{q_k} \end{psmallmatrix} v_0 = W 
    \begin{psmallmatrix} \ssvek{\varepsilon_1} & \cdots & \ssvek{\varepsilon_k} \end{psmallmatrix}^{-1} v_0 = W E^{-1} \begin{psmallmatrix} \zsvek{\varepsilon_1} \\[-1mm] \vdots \\ \zsvek{\varepsilon_k}  \end{psmallmatrix} v_0 = W\cdot E^{-1} \beta, 
    \end{equation}
    and $\det E = W^2$, we obtain
	\begin{align*}
		\Big((m+1)!\,\,\mathsf{vol}_{m+1}(S_I)\Big)^2
			&= \beta^\top
				\cdot
				\left(
					(\det E)E^{-1} \cdot \mathsf{Ext}^{[k]}_J\big( {\mathsf{adj}}\, [E^{-1}]_J\big) \cdot E^{-1}
				\right)
				\cdot
				\beta\\
			&= \frac{1}{W^{2(\ell-1)}}
				\cdot
				v_q^\top
				\cdot 
				\mathsf{Ext}^{[k]}_J\big( {\mathsf{adj}}\, [Q]_J\big)
				\cdot
				v_q\\
			&= \frac{\big((\ell-1)!\prod_{j\in J}\langle v_0,q_j\rangle\big)^2}{W^{2(\ell-1)}}
				\Big(
					\mathsf{vol}_{\ell-1}(\Delta_J)
				\Big)^2,
	\end{align*}
    where the last identity follows from Lemma \ref{lem:volume_Delta_J}.
	The claim follows, as we have
	\begin{equation} \label{eq:volratiosq}
		\left(
			\frac{\mathsf{vol}_{\ell-1}(\Delta_J)}
				{\mathsf{vol}_{m+1}(S_I)}
		\right)^2
			= \left(
				\frac{(m+1)!W^{\ell-1}}
					{(\ell-1)!\,\,\prod_{j\in J}\langle v_0,q_j\rangle}
			  \right)^2,
	\end{equation}
	and since $W$ is, up to sign, the volume of the $k$-dimensional parallel-epiped $\E$ spanned by $\varepsilon_1,\dots,\varepsilon_k$.
\end{proof}

\medskip
\noindent
\textbf{Remark:}
The derivation in \eqref{eq:vqE-1beta} shows that $\langle v_0, q_j \rangle = W \cdot (E^{-1}\beta)_j$, and therefore 
%
equation \eqref{eq:volratiosq} can be rephrased as
\[
	\frac{\vol_{\ell-1}(\Delta_J)}{\vol_{m+1}(S_I)}
		= \frac{(m+1)!}
			   {|W| (\ell-1)! \prod_{j \in J} \big| (E^{-1}\beta)_j \big|}
\]
where 
\[
	W = \big| \det ( \varepsilon_1, \ldots, \varepsilon_k )\,\big|,\qquad 
	E := \Big( 
			\langle \varepsilon_r,\varepsilon_s \rangle 
		 \Big) 
		 \in \R^{k\times k},
	\qquad\text{and}\qquad
	\beta := \begin{pmatrix} 
				\langle v_i,\varepsilon_1 \rangle \\
   				\vdots \\ 
				\langle v_i,\varepsilon_k \rangle
			 \end{pmatrix}.
\]
			   

\section{Three fundamental lemmata about the adjunct matrix}\label{sec:three_fundamental_statements}

This section provides the three lemmas used in the previous section for the proof of Theorem \ref{thm:main1}. They are all 
basic statements related to the adjunct of a square matrix.

The first statement, Lemma~\ref{lem:adjunct_lemma}, follows immediately from the multilinearity of the determinant with 
respect to its rows and columns. 

\begin{lemma}[Scaling Lemma]\label{lem:adjunct_lemma}$ $\\ 
	Let $A=(a_{ij})_{i,j \in [r]}\in\R^{r\times r}$ and $\lambda,\mu\in\R^r$. If we define the diagonal matrices
	$D_\lambda=\sum_{i\in[r]}\lambda_ie_i \cdot e_i^\top$ and $D_\mu=\sum_{i\in[r]}\mu_ie_i \cdot e_i^\top$, then
	\[
		\big(\mathsf{adj}(D_\lambda\cdot A\cdot D_\mu)\big)_{s,t}
		=
		\frac{\prod_{i,j\in [r]}\lambda_i\mu_j}
			{\lambda_t\mu_s}
		\cdot
		\big(\mathsf{adj}\, A\big)_{s,t}
		\qquad\text{for all $s,t\in[r]$.}
	\]
\end{lemma}
\begin{proof}
	From $D_\lambda\cdot A\cdot D_\mu = \left(\lambda_i\mu_ja_{ij}\right)_{i,j\in[r]}$, we obtain
    \[
		\det\left(\big(D_\lambda\cdot A\cdot D_\mu\big)^{t,s} \right)
			= 
			\frac{\prod_{i,j\in [r]}\lambda_i\mu_j}
				{\lambda_t\mu_s}
			\det (A^{t,s})
	\]
	and the claim follows.
\end{proof}

\medskip
The second statement, Lemma~\ref{lem:determinant_lemma}, follows also from the multilinearity of the
determinant in its rows and columns but is less obvious.

\begin{lemma}[Determinant Lemma]\label{lem:determinant_lemma}$ $\\ 
	Let $A = (a_{ij}) \in\R^{r\times r}$ and $T:=\begin{pmatrix}[c|c] \ssvek{\mathbb 1} & -\ID_{r-1}\end{pmatrix}\in \R^{(r-1)\times r}$. 
	Then 
	\[
		\det\big( 
			T \cdot A \cdot T^\top
		\big)
		=
		\mathbb 1_r^\top \cdot \mathsf{adj}\, A \cdot \mathbb 1_r.
	\]
\end{lemma}
\begin{proof}
	Denote row~$i$ of $A$ by~$a_i$ and column~$i$ of $T\cdot A$ by~$b_i$. Then 
	$T\cdot A\cdot T^\top\in \R^{(r-1)\times(r-1)}$ satisfies
	\[
		T \cdot A \cdot T^\top
			= \begin{pmatrix}  \zsvek{(a_1-a_2)} \\[-1mm] \vdots \\ \zsvek{(a_1 - a_r)} \end{pmatrix} \cdot T^\top
			=: \begin{pmatrix}  \ssvek{b_1}			& \cdots & \ssvek{b_r}			\end{pmatrix} \cdot T^\top
			= \begin{pmatrix}  \ssvek{(b_1-b_2)}	& \cdots & \ssvek{(b_1-b_r)}	\end{pmatrix}.
	\]
	Since the determinant is linear all columns and
    vanishes if two columns coincide, 
    we obtain
	\[
		\det
			\begin{pmatrix}  \ssvek{(b_1-b_2)}	& \cdots & \ssvek{(b_1-b_r)}	\end{pmatrix}
		=
		(-1)^{r-1}
			\left( \det
				\begin{pmatrix}  \ssvek{b_2}			& \cdots & \ssvek{b_r}			\end{pmatrix}
				-
				\sum_{s=2}^r
					\det \begin{pmatrix}  \ssvek{b_2}			& \cdots & \ssvek{\widetilde b_s} & \cdots & \ssvek{b_r}			\end{pmatrix}
			\right)
	\]
	with $\widetilde b_s = b_1$ for $2\leq s \leq r$. Consequently,
    \begin{equation} \label{eq:detrel}
         \det
			\begin{pmatrix}  \ssvek{(b_1-b_2)}	& \cdots & \ssvek{(b_1-b_r)}	\end{pmatrix}
		= (-1)^{r-1} \sum_{s=1}^r (-1)^{s-1} \det \begin{pmatrix}  \ssvek{b_1}			& \cdots & \ssvek{\widehat{b_s}} & \cdots & \ssvek{b_r}			\end{pmatrix},
    \end{equation}
    where $\widehat{b_s}$ means that this column is deleted.
	\noindent
	Introducing
    \[
		C_s := \begin{pmatrix}  \ssvek{b_1}			& \cdots & \ssvek{\widehat{b_s}} & \cdots & \ssvek{b_r}			\end{pmatrix} = \begin{pmatrix}  \zsvek{(a_1^s-a_2^s)} \\[-1mm] \vdots \\ \zsvek{(a_1^s - a_r^s)} \end{pmatrix}
			 \in \R^{(r-1)\times r}, 
	\]
 where $a_r^s$ is the row vector $a_r$ with the $s$-th entry removed. Similarly to \eqref{eq:detrel}, but now considering rows instead of columns, we obtain 
    \begin{multline*} 
    \det C_s = \det \begin{pmatrix}  \zsvek{(a_1^s-a_2^s)} \\[-1mm] \vdots \\ \zsvek{(a_1^s - a_r^s)} \end{pmatrix} = (-1)^{r-1} \sum_{t=1}^r (-1)^{t-1} \det \begin{pmatrix}  \zsvek{a_1^s} \\[-1mm] \vdots \\ \zsvek{\widehat{a_t^s}} \\[-1mm] \vdots \\[-1mm] \zsvek{a_r^s} \end{pmatrix} \\ = (-1)^{r-1} \sum_{t=1}^r (-1)^{t-1} \det(A^{t,s}),
    \end{multline*}
    where $\widehat{a_t^s}$ means that this row is deleted. This shows
	\begin{align*}
		\det\big( T	\cdot A \cdot T^\top \big)
			&= (-1)^{r-1} \sum_{s=1}^r (-1)^{s-1} \det C_s \\
            & = (-1)^{r-1} \sum_{s=1}^r (-1)^{s-1} \left( (-1)^{r-1} \sum_{t=1}^r (-1)^{t-1} \det (A^{t,s}) \right) \\
 	        &= \sum_{s,t=1}^r (-1)^{s+t} \det (A^{t,s})
            = \sum_{s,t=1}^r \big( \mathsf{adj}\, A \big)_{s,t} = \mathbb 1_r^\top \cdot \mathsf{adj}\, A \cdot \mathbb 1_r.
    \end{align*}
	and the claim follows.
\end{proof}

\noindent
As $A^{-1}=\frac{1}{\det(A)}\cdot\mathsf{adj}\, A$ for all $A\in\operatorname{GL}(r)$, Lemma~\ref{lem:determinant_lemma} 
implies the following corollary.
\begin{corollary}$ $\\
	Let $A\in\operatorname{GL}(r)$ and $T:=\begin{pmatrix}[c|c] \ssvek{\mathbb 1} & -\ID_{r-1}\end{pmatrix}\in \R^{(r-1)\times r}$. Then 
	\[
		\det\big( T \cdot A \cdot T^\top \big)
		=
		\det(A) \mathbb 1_r^\top \cdot A^{-1} \cdot \mathbb 1_r.
	\]
\end{corollary}

\medskip
The third statement, Proposition~\ref{prop:extension_lemma}, relates $M$ and its adjoint $N:=\mathsf{adj}\, M$ 
to square matrices $\ext{[r]}{I} (\mathsf{adj}\, [M]_{I})$ and $\ext{[r]}{J} (\mathsf{adj}\, [N]_{J})$ that extend
the adjoints of $[M]_{I}$ and $[N]_{J}$ which are restrictions of $M$ and $N$. 

\begin{proposition}[Extension Lemma]\label{prop:extension_lemma}$ $\\ 
	Let $M\in\R^{r\times r}$ and $I\in\binom{[r]}{a}$ and consider $N:=\mathsf{adj}\, M$ and 
	$J:=[r]\setminus I$. Then
	\[
		(\det M)^{r-a}(\det\, [M]_I) 
		\ID_r
		=
		(\det M)^{r-a} 
			\big(\ext{[r]}{I} (\mathsf{adj}\, [M]_{I}) \big)
			\cdot
			M
		+ (\det M) 
			N
			\cdot
			\big( \ext{[r]}{J} (\mathsf{adj}\, [N]_{J}) \big).	
	\]
	\noindent
	In particular, if $\det M \neq 0$ then
	\[
		\big(\det[M]_I\big)M^{-1}
		-
		\ext{[r]}{I}(\mathsf{adj}\, [M]_{I})
		=
		\frac{1}{(\det M)^{r-a-1}}
			N
			\cdot
			\big(
				\ext{[r]}{J} (\mathsf{adj}\, [N]_{J}) 
			\big)
			\cdot
			M^{-1}
	\]
\end{proposition}

We specialise Proposition~\ref{prop:extension_lemma} to $I=[a]$ and $J=[r]\setminus[a]$, and the proposition follows from Lemma \ref{prop:extension_lemma} below, by sorting a basis appropriately. Note that in this special case, we have for $X\in \R^{a\times a}$ and $Y\in\R^{(r-a)\times(r-a)}$,
\[
	\ext{[r]}{[a]}(X)
		= 	\left(\,\begin{matrix}[c|c] 
						X		  & \mathbb 0\\ \hline 
						\mathbb 0 & \mathbb 0
			\end{matrix}\,\right)
	\qquad\text{and}\qquad
	\ext{[r]}{[r]\setminus[a]}(Y)
		= \left(\,\begin{matrix}[c|c] 
						\mathbb 0 & \mathbb 0\\ \hline 
						\mathbb 0 & Y	
			\end{matrix}\,\right).
\]

\begin{lemma}[Sorted Extension Lemma]\label{lem:sorted_extension_lemma}$ $\\
	Let $M\in\R^{r\times r}$ and $0\le a \le r$. Let $M_1 \in \R^{a \times a}$ and 
	$N_2 \in \R^{(r-a)\times(r-a)}$ be given by
	\[
		M =: \left(\,\begin{matrix}[c|c] 
						M_1 & *\\ \hline 
						* 	& * 
			\end{matrix}\,\right)
		\qquad\text{and}\qquad 
		N:=\mathsf{adj}(M)
			=:\left(\,\begin{matrix}[c|c] 
						*	& *\\\hline 
						* 	& N_2 
			 \end{matrix}\,\right).
	\]
	Then
	\begin{equation} \label{eq:sortres}
		(\det M)^{r-a}(\det M_1) 
		\ID_r
		=
		(\det M)^{r-a} 
			\left(\,\begin{matrix}[c|c] 
						\mathsf{adj}\, M_1 & \mathbb 0\\ \hline 
						\mathbb 0 & \mathbb 0
			\end{matrix}\,\right)
			\cdot
			M
		+
		(\det M) 
			N
			\cdot
			\left(\,\begin{matrix}[c|c] 
						\mathbb 0 & \mathbb 0\\ \hline 
						\mathbb 0 & \mathsf{adj}\, N_2	
			\end{matrix}\,\right).
	\end{equation}
\end{lemma}

\begin{proof}$ $
	
	We first prove the degenerate situations $a=r$ and $a=0$. If $a=r$ then $M_1=M$ and the claim reduces to the standard 
	result $(\det M) \ID_r	= \mathsf{adj}\, M \cdot M + (\det M) N \cdot \mathbb 0$. If $a=0$ then $N=N_2\in\R^{r\times r}$ 
	and $N=\mathsf{adj}\, M$ which implies $(\det M)(\det N)
			= (\det M)^r$.
	Thus we get the claim by
	\[
		(\det M)^{r-0} \ID_r
			= (\det M)^{r-0} \mathbb 0 \cdot M
				+ (\det M)N\cdot \mathsf{adj}\, N.
	\]

	\medskip
	If $0< a <r$ then 
	\[
		M=\left(\,\begin{matrix}[c|c]
			 		M_1 & A\\ \hline 
					B	& M_2 
		  \end{matrix}\,\right)
		\qquad\text{and}\qquad
		N=\mathsf{adj}\, M=\left(\,\begin{matrix}[c|c] 
					N_1 & C\\ \hline 
					D 	& N_2 
		  \end{matrix}\,\right)
	\]
	yields
	\[
		\left(\,\begin{matrix}[c|c] 
			\mathsf{adj}\, M_1 & \mathbb 0\\ \hline 
			\mathbb 0 & \mathbb 0
		\end{matrix}\,\right)
		\cdot
		M
			=
			\left(\,\begin{matrix}[c|c] 
				(\det M_1)\ID_a	& \mathsf{adj}\, M_1\cdot A\\ \hline
				\mathbb 0 				& \mathbb 0
			\end{matrix}\,\right)
	\]
	and
	\[
		N
		\cdot
		\left(\,\begin{matrix}[c|c]
			\mathbb 0 & \mathbb 0\\ \hline
			\mathbb 0 & \mathsf{adj}\, N_2
		\end{matrix}\,\right)
		=
		\left(\,\begin{matrix}[c|c]
			\mathbb 0 & C\cdot\mathsf{adj}\, N_2\\ \hline
			\mathbb 0 & (\det N_2) \ID_{r-a}
		\end{matrix}\,\right).
	\]
	Hence $\alpha,\beta\in \R$ imply
	\begin{equation} \label{eq:sortextproof}
		\alpha
		\left(\,\begin{matrix}[c|c]
			\mathsf{adj}\, M_1	& \mathbb 0\\ \hline 
			\mathbb 0 			& \mathbb 0
		\end{matrix}\,\right)
		\cdot
		M
		+
		\beta
		N
		\cdot
		\left(\,\begin{matrix}[c|c]
			\mathbb 0 & \mathbb 0\\ \hline
			\mathbb 0 & \mathsf{adj}\, N_2
		\end{matrix}\,\right)
		=
		\left(\,\begin{matrix}[c|c] 
			\alpha(\det M_1) \ID_a	& \alpha\mathsf{adj}\, M_1\cdot A + \beta C\cdot\mathsf{adj}\, N_2\\ \hline
			\mathbb 0						& \beta(\det N_2) \ID_{r-a}
			\end{matrix}\,\right).		
	\end{equation}
	If we set $\alpha:=(\det M)^{r-a}$ and $\beta:=\det M$ then the left hand side of \eqref{eq:sortextproof} agrees with the right hand side of \eqref{eq:sortres} and it remains to prove the following claims
	\begin{compactenum}[i)]
		\item $(\det M)(\det N_2) = (\det M_1)(\det M)^{r-a}$;\label{claim_i}
		\item $(\det M)C \cdot \mathsf{adj}\, N_2 = - (\det M)^{r-a} \mathsf{adj}\, M_1 \cdot A$.\label{claim_ii}
	\end{compactenum}
	Before proving these claims, we notice that
	\[
		(\det M)	
		\left(\,\begin{matrix}[c|c]
				\ID_{a}	& \mathbb 0\\ \hline
				\mathbb 0		& \ID_{r-a} 
		\end{matrix}\,\right)
		= 
		M \cdot \mathsf{adj}\, M
		= M \cdot N =
		\left(\,\begin{matrix}[c|c]
			M_1\cdot N_1 + A\cdot D	& M_1\cdot C + A \cdot N_2\\ \hline
			B\cdot N_1 + M_2\cdot D	& B\cdot C + M_2\cdot N_2 
		\end{matrix}\,\right)
	\]
	implies
	\[
		A\cdot N_2 = -M_1 \cdot C
		\qquad\text{and}\qquad
		M_2\cdot N_2 = (\det M)\cdot \ID_{r-a} - B\cdot C.
	\]
	Claim~\ref{claim_i}) is a consequence of the following computation: 
	\begin{align*}
		(\det M)  (\det N_2)
			&= \det\left(
						\left(\,\begin{matrix}[c|c]
							M_1	& A\\ \hline
							B	& M_2 
						\end{matrix}\,\right)
						\cdot
						\left(\,\begin{matrix}[c|c]
							\ID_a	& \mathbb 0\\ \hline
							\mathbb 0		& N_2 
						\end{matrix}\,\right)
				   \right) \\
			&= \det
			   \left(\,\begin{matrix}[c|c]
					M_1	& A \cdot N_2\\\hline 
					B	& M_2 \cdot N_2
				\end{matrix}\,\right) \\
			&= \det
			   \left(\,\begin{matrix}[c|c]
					M_1	& -M_1 \cdot C\\ \hline 
					B	& (\det M)\cdot\ID_{r-a} - B \cdot C
				\end{matrix}\,\right) \\
			&= \det \left( 
						\left(\,\begin{matrix}[c|c]
							M_1	& \mathbb 0\\ \hline
							B	& \ID_{r-a}
						\end{matrix}\,\right)
						\cdot
						\left(\,\begin{matrix}[c|c]
							\ID_a	& -C\\ \hline
							\mathbb 0		& (\det M)\cdot \ID_{r-a}
						\end{matrix}\,\right)
				\right)\\
			&= \left( \det M_1 \right) \cdot \left( \det M \right)^{r-a}.
	\end{align*}
	
	\medskip
	\noindent
	\medskip
	We now derive Claim~\ref{claim_ii}) from Claim~\ref{claim_i}). First, we additionally assume $M_1\in\operatorname{GL}(a)$:
	\begin{align*}
		(\det M)
		C
		\cdot
		\mathsf{adj}\, N_2
			&= (\det M)
				\Big(
					M_1^{-1}
					\cdot
					M_1 
					\cdot 
					C
				\big)
				\cdot
				\mathsf{adj}\, N_2\\
			&= -\frac{\det M}{\det M_1} 
			   	\mathsf{adj}\, M_1
				\cdot
				(A\cdot N_2)
				\cdot
				\mathsf{adj}\,N_2\\
			&= -\frac{(\det M)(\det N_2)}{\det M_1}
				\mathsf{adj}\, M_1
				\cdot
				A\\
			&= -(\det M)^{r-a}
				\mathsf{adj}\, M_1
				\cdot
				A.
	\end{align*}
	We now prove $(\det M) C \cdot \mathsf{adj}\, N_2 = -(\det M)^{r-a} \mathsf{adj}\, M_1 \cdot A$ for 
	$M_1\not\in\operatorname{GL}(a)$, that is, if $\det M_1=0$. Clearly, 
	$p(\epsilon) := \det(M_1 + \epsilon\cdot\ID_a) \in \R[\epsilon]$ is a polynomial of degree~$0<a\leq r$.
 In particular, $p$ has at most $a$ distinct zeroes. Since $p(0)=0$, 
	there exists $\epsilon_0>0$ such that $M_1+x\cdot\ID_a\in \operatorname{GL}(a)$ for all 
	$0<x<\epsilon_0$. We now set
	\[
		M_x
		:= 
		\left(\,\begin{matrix}[c|c]
			M_1+x\cdot\ID_a	& A\\ \hline
			B						& M_2
		\end{matrix}\,\right)
		\qquad\text{and}\qquad
		\widetilde N 
		:= \mathsf{adj}(M_x)
		=
		\left(\,\begin{matrix}[c|c]
			\widetilde N_1	& \widetilde C\\ \hline
			\widetilde D	& \widetilde N_2
		\end{matrix}\,\right)
	\]
	and obtain
	\[
		(\det M_x) \widetilde C \cdot \mathsf{adj}(\widetilde N_2) 
			= 
			-(\det M_x)^{r-a} \mathsf{adj}(M_1+x\cdot\ID_a) \cdot A
		\qquad\text{for all $0<x<\varepsilon_0$}.
	\]
	Continuity now implies 
	\[
		(\det M_x) \widetilde C \cdot \mathsf{adj}(\widetilde N_2) 
			= 
			-(\det M_x)^{r-a} \mathsf{adj}(M_1) \cdot A
		\qquad\text{for $x=0$.}
	\] 
\end{proof}

\section{Vertex transitive polytopes}
\label{sec:vertextranspoly}

In this section, we return to Izmestiev's Colin de Verdi{\`e}re Matrix and investigate the case of vertex transitive $k$-polytopes $Q \subset \mathbb{R}^k$. Let $w_1,\dots,w_n$ denote the vertices of this polytope. Then the arithmetic mean of the vertices
$$ \widehat w(Q) = \frac{1}{n}(w_1 + \cdots + w_n) $$
lies in $Q$, by convexity. We call $\widehat w(Q)$ the \emph{centre} of the polytope $Q$. For a vertex transitive polytope, $\widehat w(Q)$ is a fixed-point of the full isometry group $\Iso(Q)$ of $Q$. By applying a translation, if necessary, we can assume that $\widehat w(Q) = \0$. This implies that $\Iso(Q)$ is a finite subgroup of $O(k)$ and $\|w_s\| = \| w_t\|$ for all $s,t\in[n]$. 
Henceforth we always implicitely assume $\widehat w(Q) = \0$ when we refer to vertex transitive polytopes $Q$.

To have enough flexibility, we will often consider a subgroup $\Gamma \subset \Iso(Q)$, which acts still transitively on the vertices of $Q$. We say that $Q$ is vertex transitive with respect to $\Gamma \subset O(k)$.
	
The $\Gamma$-action induces an action on 
the $r$-faces of $Q$ for all $r \in [k]$ and an equivalence relation on the set of $r$-faces, whose equivalence classes are the $\Gamma$-orbits. In particular, we have for any element in $\Gamma \subset O(k)$ that maps $w_s$ to $w_{s'}$ and its neighbor $w_t$ to $w_{t'}$ the following identity for the entries of the Colin de Verdi{\`e}re Matrix $M$ from Theorem~\ref{thm:ivan}:
\begin{equation} \label{eq:Ms't'st}
	 M_{s't'} 
	 	= - \frac{\vol_{k-2}(F_{s't'})} {\Vert w_{s'} \Vert \cdot \Vert w_{s'} \Vert \sin \theta_{s't'}} 
		= - \frac{\vol_{k-2}(F_{st})} {\Vert w_s \Vert \cdot \Vert w_t \Vert \sin \theta_{st}} 
		= M_{st}. 
\end{equation}
In other words, the entries of $M$ corresponding to $\Gamma$-equivalent edges of~$Q$ coincide. Consequently, we have for any pair $s,s' \in [n]$, $\sum_{t\neq s} M_{st}=\sum_{t'\neq s'} M_{s't'}$. Moreover, it can be seen from \eqref{eq:Mstorig} that vertex transitivity implies that all diagonal entries $M_{ss}$ coincide. In particular, the values in one row of $M$ determine the entries of all other rows in the vertex transitive case.

For vertex transitive $k$-polytopes we therefore have: the diagonal entries~$M_{ss}$ of the (symmetric) Colin de Verdi\`ere matrix~$M$ of Theorem~\ref{thm:ivan} are all equal to some constant~$\gamma \in \mathbb{R}$ and summing the entries of any column or row yields another constant~$\delta \in \mathbb{R}$. We have $\gamma > \delta$ since all off-diagonals of $M$ are non-positive, and some of them are strictly negative. In the next subsection, we introduce a closely related stochastic matrix for these vertex transitive polytopes. 


\subsection{Random walk matrices and the spectral gap}

Assume that $Q \subset \mathbb{R}^k$ is a vertex transitive $k$-polytope with Colin de Verdi{\`e}re Matrix $M$ given in Theorem \ref{thm:ivan}. We can associate to $M$ a family of symmetric matrices
\begin{equation} \label{eq:TtauM}
	T_\tau(M) := \frac{1}{\tau-\delta}\big( \tau\mathsf{Id} - M  \big)
\end{equation}
for every $\tau\geq\gamma$. Then $T_\tau(M)$ is a doubly stochastic symmetric matrix that has the same eigenspaces as~$M$. Any such matrix $T = T_\tau(M)$ describes a random walk on the vertex-edge graph $G$ of $Q$ with laziness $\frac{\tau-\gamma}{\tau-\delta}$ at each vertex of $G$. The choice $\tau = \gamma$ leads to a random walk without laziness. We will henceforth denote the stochastic matrix $T_\gamma(M)$ without laziness 
simply by $T = T(Q)$.

Since $T$ is symmetric, its spectrum is real. The highest eigenvalue of $T$ is always $\lambda_0(T)=1$ with constant eigenvector $\mathbb 1_n$. Of particular interest is the second largest eigenvalue $\lambda_1(T)$ of $T$, which has the variational characterisation
$$ \lambda_1(T) = \sup_{v \bot \mathbb 1_n, v \neq 0} \frac{v^\top Tv}{v^\top v}. $$
Regarding $\lambda_1(T)$, we have the following reformulation of Corollary \ref{cor:specrep} in this particular case.

\begin{corollary}
    Let $Q \subset \mathbb{R}^k$ be a vertex transitive $k$-polytope with vertices $w_1,\dots,w_n$ and $\phi_1,\dots,\phi_k$ be defined via \eqref{eq:wphi}. Then the vectors $\phi_j \in \mathbb{R}^n$ are a basis of the $T(Q)$-eigenspace to the second largest eigenvalue 
    $$ \lambda_1(T(Q)) = \frac{\gamma}{\gamma-\delta} < 1. $$ 
\end{corollary}

\begin{proof}
    We know from Corollary \ref{cor:specrep} that $\phi_1,\dots,\phi_k$ are a basis of $\ker M$ and therefore also of the eigenspace of the matrix $T= T(Q) = T_\gamma(M)$ to the eigenvalue
    $$ \lambda = \frac{\gamma}{\gamma-\delta}, $$
    by \eqref{eq:TtauM}. Note also that \eqref{eq:TtauM} and $\gamma > \delta$ implies that the (second) largest eigenvalue of $T$ corresponds to the (second) smallest eigenvalue of $M$. Since $M$ is a Colin de Verdi{\`e}re Matrix, its second smallest eigenvalue is $0$, which implies that $\lambda$ is the second largest eigenvalue of $T$.
\end{proof}    

Let us now briefly recall the notion of a Cayley graph of a finitely generated group $\Gamma$ (see \cite{C-1878} for its origin and, e.g., \cite{B-95} for a modern introduction). Let $S = \{s_1,\dots,s_m\}$ be a symmetric set of non-trivial generators, that is, none of the elements $s_j$ are the identity $e \in \Gamma$, and for any $s \in S$ we have also $s^{-1} \in S$. The associated Cayley graph ${\rm{Cay}}(\Gamma,S)$ is then defined as follows: its vertex set agrees with $\Gamma$, and there is an edge between $g, g' \in \Gamma$ if and only if there exists $s_j \in S$ such that $g' = g s_j$. Note that ${\rm{Cay}}(\Gamma,S)$ is a simple graph, that is, it has no loops and no multiple edges. We have the following standard fact.

\begin{proposition} \label{prop:polycayleygraph}
    Let $Q \subset \mathbb{R}^k$ be a $k$-polytope such that $\Gamma \subset O(k)$ acts simply transitively on the vertices $w_1,\dots,w_n$ of $Q$. Assume that $w_1,\dots,w_m$ are the neighbours of $w_{m+1}$ and let $w_j = s_j w_{m+1}$ for all $j \in [n]$. 

    Then $s_1,\dots,s_m$ is a symmetric set of generators of $\Gamma$ and the vertex-edge graph of $Q$ is canonically isomorphic to ${\rm{Cay}}(\Gamma,S)$ with $S = \{s_1,\dots,s_m \}$ via the identification $w_j \mapsto s_j$.
\end{proposition}

\begin{proof}
    The symmetry of $S$ can be seen as follows: Let $s \in S$, that is, there exists $i \in [m]$ with $s=s_i$. Then $\{w,w_i\}$ is an edge of $Q$ and, since the isometry $s^{-1} \in \Gamma$ permutes the edges of $Q$, 
    $\{s^{-1} w, s^{-1} w_i = w \}$ 
    is again an edge of $Q$, and there must exist $j \in [m]$ such that $\{w,w_j\} = \{w,s^{-1} w\}$. The $1:1$-correspondence between the vertices of $Q$ and the elements of $\Gamma$ implies that $s^{-1} = s_j$. This shows that $s^{-1} \in S$.  
    
    The set $S$ generates $\Gamma$ since there is a path from $w_{m+1}$ to any other vertex $w_t = s_t w_{m+1}$, $t \in [n]$, along edges of $Q$, and this path translates into a representation of $s_t$ in terms of the elements $s_j$, $j \in [m]$. Note that the edges $\{w_{m+1},w_j\}$, $j \in [m]$, of $Q$ are in $1:1$-correspondence to the edges $\{e,s_j\}$, $j \in [m]$, in $G$ incident to $e$, and that this correspondence carries over to every other vertex via the transitive action of $\Gamma$ both on the vertices of $Q$ and the vertices of $G$.
\end{proof}

Let us now return to the 
stochastic matrix $T(Q)$. Note that our definition of $T(Q)$ requires the existence of a finite group $\Gamma \subset O(k)$, acting transitively on the vertices of $Q$. Let us assume again that $\Gamma$ acts simply transitively on the vertices of $Q$. We sort the neighbours $w_1,\dots,w_m \in Q$ of the vertex $w_{m+1} \in Q$ and the corresponding generators $s_1,\dots,s_m$ in Proposition \ref{prop:polycayleygraph} in such a way that the first $\nu$ group elements $s_1\dots,s_\nu \in \Gamma$ have order $2$, the next $2\mu$ group elements $s_{\nu+1},\dots,s_{\nu+2\mu}$ have orders $\ge 3$, and that $s_{\nu+2j} = s_{\nu+2j-1}^{-1}$ for $j \in [\mu]$. The relation \eqref{eq:MIvan} of Izmestiev's Colin de Verdi{\`e}re Matrix $M$ means that we have
\begin{equation} \label{eq:stochrelvertices} 
\gamma w_{m+1} = (\gamma-\delta) \left( \left( \sum_{j=1}^\nu x_j w_j \right) + \left( \sum_{j=1}^{\mu} x_{\nu+j} (w_{\nu+2j-1} + w_{\nu+2j}) \right) \right), 
\end{equation}
that is, the vertex $w_{m+1}$ of $Q$ is parallel to a weighted sum of its neighbouring vertices $w_1,\dots,w_m$
with suitable weights $x_1,\dots,x_{\nu+\mu} > 0$ satisfying
$$
\left( \sum_{j=1}^\nu x_j \right) + 2 \left( \sum_{j=\nu+1}^{\nu+\mu} x_j \right) = 1. 
$$
By the relation \eqref{eq:TtauM}, the
$(m+1)$-st row of the associated stochastic matrix $T(Q)$ is then given by
\begin{equation} \label{eq:(m+1)-row-TQ} 
(x_1,\dots,x_\nu,x_{\nu+1},x_{\nu+1},\dots,x_{\nu+\mu},x_{\nu+\mu},0,\dots,0), 
\end{equation}
and every other row of $T(Q)$ 
is obtained by permuting these entries according to the group action (cf. \eqref{eq:Ms't'st}). Therefore it makes sense to introduce the simplex of probability weights
\begin{equation} \label{eq:simpDelGamm} 
\Delta_\Gamma = \set{x_1,\dots,x_{\nu+\mu} \ge 0}{\left( \sum_{j=1}^\nu x_j \right) + 2\left( \sum_{j=\nu+1}^{\nu+\mu} x_j\right) = 1},
\end{equation}
collecting the weights in the stochastic matrix $T(Q)$.

Let us mention that under the transition from the Colin de Verdi{\`e}re Matrix $M$ to the matrix $T(Q)$, we lose the information about the diagonal entries of $M$. This means that we cannot reconstruct $M$ from the stochastic matrix $T(Q)$.



A natural quantity arising in relation to the spectrum is the \emph{spectral gap}, that is, the value $1 - \lambda_1(T(Q))$, where $\lambda_1$ denotes the second largest eigenvalue. 
The spectral gap is known to be a measure for exponential mixing of the random walk on the associated vertex-edge graph. An estimate for the spectral gap of the original matrix $M$ and interest in its maximality was also mentioned in \cite[Subsection 3.2]{Izm-10}. 

Every vertex transitive $k$-polytope $Q \subset \mathbb{R}^k$ with respect to $\Gamma \subset O(k)$ gives rise to the following variation: Let  $S^{k-1}(r) \subset \mathbb{R}^k$ be the sphere of radius $r > 0$ centred at the origin and containing all the vertices of $Q$. Then we have $Q = \conv(\Gamma w_0)$, where $w_0 \in Q$ is any vertex of $Q$, and a \emph{variation} of $Q$ is given by  
\begin{equation} \label{eq:variation}
S^{k-1}(r) \ni w \mapsto Q_w = \conv(\Gamma w). 
\end{equation}
Instead of considering a variation of a given polytope, one can also start with a finite group $\Gamma \subset O(k)$ with the property that $Q_w = \conv(\Gamma w)$ is a $k$-polytope for some $w \in \mathbb{R}^k$, and then investigate the function $S^{k-1}(r) \ni w \mapsto \lambda_1(T(Q_w)) \in [-1,1]$ for some $r > 0$.
Note that it is sufficient to restrict the domain of this function to a sphere $S^{k-1}(r)$, since $Q_w$ and $Q_{\mu w}$ with $\mu > 0$ are homothetic and $\lambda_1$ rescales accordingly by a fixed positive factor.
The function $\lambda_1$ is not everywhere continuous on $S^{k-1}(r)$, due to the fact that the combinatorial types of the polytopes $Q_w$ may change under the variation. This happens, for example, if $\Gamma$ acts simply transitively on the orbit $\Gamma w$ but not on the orbit $\Gamma w'$. In this case, the number of vertices of $Q_{w'}$ is smaller than the number of vertices of $Q_w$. The subset of $w \in \mathbb{R}^k$ for which $\Gamma$ does not act simply transitively on  $\Gamma w \subset \mathbb{R}^k$ is a finite union of lower dimensional linear subspaces (the eigenspaces of the non-trivial matrices in $\Gamma$ to the eigenvalue $1$). This means that $\Gamma$ has simply transitive $\Gamma$-orbits in an open subset $S^{k-1}_{\rm{reg}}(r) \subset S^{k-1}(r)$ of full measure. We refer to $S^2_{\rm{reg}}(r)$ as the \emph{regular subset} of the $\Gamma$-action on $S^2(r)$. While the number of vertices of $Q_w$ does not change within $S^{k-1}_{\rm{reg}}(r)$, the combinatorial type may still change within 
this subset: for example, in the case $k=3$, families of certain polygonal $2$-faces of $Q_w$, which lie in different affine planes for some $w \in S_{\rm{reg}}^{k-1}(r)$, may become parallel within $S_{\rm{reg}}^{k-1}(r)$ and therefore part of a single plane under the variation, and they
may therefore merge into a single new face, with the effect that certain neighbours of a vertex $w$ may no longer be its neighbours after this process. In this case, the corresponding entry of Izmestiev's Colin de Verdi{\`e}re matrix shrinks to zero but this does not affect the continuity of $\lambda_1$
at these points. For background information about combinatorics and geometry of polytopes, we refer readers to \cite{HRGZ-97}.


In the next subsections, we illustrate the variation \eqref{eq:variation} by some concrete examples. Before doing so, we introduce the natural maps
\begin{equation} \label{eq:PsiDelta} 
\Psi_\Delta: S_{\rm{reg}}^{k-1}(r) \to \Delta_\Gamma 
\end{equation}
and
\begin{equation} \label{eq:Psilambda}
\lambda_1: S^{k-1}(r) \to [-1,1] 
\end{equation}
associated to this variation.
For a given $w \in S_{\rm{reg}}^{k-1}(r)$, we define $\Psi_\Delta(w)$ to be the weights of the stochastic matrix $T(Q_w) \in \mathbb{R}^{n \times n}$ in $\Delta_\Gamma$. Here $n$ is the number of vertices of the associated polytope $Q_w$, which remains constant for all $w \in S_{\rm{reg}}^{k-1}(r)$. For a general $w \in S^{k-1}(r)$, we set $\lambda_1(w)$ equal to $\lambda_1(T(Q_w))$, where  
the size of $T(Q_w)$ is again equal to the number of vertices of $Q_w$ and may be smaller than $n$.

\subsection{Example: Coxeter groups and permutahedra}
\label{subsec:coxperm}


In this subsection, we use Theorem \ref{thm:main1} to express Izmestiev's geometrically motivated Colin de Verdi{\'e}re Matrix $M$ of Coxeter permutahedra $Q$ in algebraic terms of the underlying Coxeter system (the roots and the Schl\"afli matrix). Then we employ the agreement of the associated stochastic matrix $T(Q)$ with the stochastic matrix $P_X$, $X \in \Delta_\Gamma$ introduced in \cite{IP-13}. In particular, the constant corank property of Izmestiev's matrix implies that the multiplicity of $\lambda_1(P_X)$ is constant on $S_{\rm{reg}}^{k-1}(r)$. This settles an open question posed in Remark  1.11(a) of \cite{IP-13} and allows us to derive a connection between equilaterality of the permutahedron $Q$ and the minimality of $\lambda_1(T(Q))$. Such a connection was established in \cite{IP-13} only for the particular Coxeter groups $A_3, B_3, H_3$. 

Let $\Gamma$ be an irreducible finite Coxeter group $\Gamma$ of rank $k$, given in the generator-relation presentation by
$$ \Gamma = \langle S = \{s_1,\dots,s_k\}\, \mid\, (s_is_j)^{m_{ij}} = e \rangle. $$ 
As usual, we require $m_{ij} = m_{ji} \ge 2$ for $i \neq j$ and $m_{ii} = 1$. For brackground of the following exposition, we refer readers to Section 1.4 of \cite{IP-13}. Let $\Gamma \hookrightarrow O(k)$ be the geometric realisation of $\Gamma$ as a finite reflection group, $\sigma_j \in O(k)$ be the reflection corresponding to the generator $s_j$, and $n_1,\dots,n_k$ be the associated unit normals of the corresponding mirrors such that any pair of unit normals $n_i,n_j$, $i \neq j$ has an angle $\ge \pi$. Let $N = (\langle n_i,n_j \rangle)_{ij \in [k]}$ be the associated Schl\"afli Matrix and assume that 
\begin{equation} \label{eq:V} 
V = \det \begin{psmallmatrix} \ssvek{n_1} & \cdots & \ssvek{n_k} \end{psmallmatrix} > 0 
\end{equation}
(this can always be arranged by a possible renumbering of the reflections $\sigma_j$).
Then
$$ \sigma_i(p) = p - 2\langle p,n_i \rangle n_i. $$
Let
\begin{equation} \label{eq:pj}
p_j = (-1)^{j-1} n_1 \times \cdots \times \widehat n_j \times \cdots \times n_k.
\end{equation}
It follows from $\langle p_i,p_j \rangle = V \delta_{ij}$ that 
\begin{equation} \label{eq:pipj}
\langle p_i,p_j \rangle = V^2 (N^{-1})_{ij}.
\end{equation}
Moreover,
$$ \mathcal{F} = \set{\sum_{j=1}^k \alpha_j p_j}{\alpha_1,\dots,\alpha_k > 0} \subset \mathbb{R}^k $$
is an (open) fundamental domain of the $\Gamma$-action on $\mathbb{R}^k$. For any $w \in \mathcal{F}$, the corresponding polytope $Q_w = \conv(\Gamma w)$ is simple and called a \emph{permutahedron} of the Coxeter group $\Gamma$, which acts simply transitively on the vertices of $Q_w$. The vertex-edge graph $G$ of $Q(v_0)$ agrees with the Cayley Graph ${\rm{Cay}}(\Gamma,S)$, by Proposition \ref{prop:polycayleygraph}. For more details, we refer the readers to \cite[Section 1.4]{IP-13}. Considering permutahedra corresponding to a spherical $\Gamma$-fundamental domain, we have the following result.

\begin{corollary} \label{cor:CoxMdescr}
    Let $\Gamma$ be an irreducible finite Coxeter group of rank $k$ with geometric realisation $\Gamma \hookrightarrow O(k)$ as a Euclidean reflection group, $\sigma_1,\dots,\sigma_k$ and $N$ be the reflections and Schl\"afli matrix associated to the simple roots $n_1,\dots,n_k \in \R^k$ introduced before. Let
    $$
    \mathcal{F}_0 = \set{p = \sum_{j=1}^k \alpha_j p_j}{p \in S^{k-1}(1), \alpha_1,\dots,\alpha_k > 0 } 
    $$
    be a spherical fundamental domain of the $\Gamma$-action with $p_1,\dots,p_k \in \R^k$ defined in \eqref{eq:pj}. 
    
    Let $w = \sum_j \alpha_j p_j \in \mathcal{F}_0$ be fixed, $\alpha = (\alpha_1,\dots,\alpha_k)^\top$, and 
    $Q_w = \conv(\Gamma w)$ the associated permutahedron with vertices $w_1,\dots,w_n \in S^{k-1}(1)$. We assume the enumeraton is chosen such that $w_{k+1} = w$ and $w_j = \sigma_j(w)$ for $j \in [k]$. Then the non-zero off-diagonals of the $(k+1)$-st row of the Colin de Verdi{\'e}re Matrix $M$ are given by
    \begin{equation} \label{M:permut}
       M_{k+1,i} = - \frac{1}{2 V^{k+1} (k-2)!\cdot
       \alpha_i \cdot \prod_{j \in [k]\setminus \{i\}} |(N^{-1} \alpha)_j |},
    \end{equation}
    with $V > 0$ given in \eqref{eq:V}.
\end{corollary}

Note that \eqref{M:permut} determines also the diagonal entry of the $(k+1)$-st row of $M$ by the relation \eqref{eq:MIvan} and the entries of all other rows of $M$ by vertex transitivity via the relation \eqref{eq:Ms't'st}.

\begin{proof}
    Let $w = \sum_j \alpha_j p_j \in \mathcal{F}_0$ and $Q_w = \conv(\Gamma w)$ with vertices $w_1,\dots,w_n$ and $n = |\Gamma|$. Moreover, if we choose an enumeration of these vertices as given in the corollary, then the vectors $p_j$ are outer facet normals of the facets $F_j = \conv\{ w_1, \dots, \widehat w_j, \dots, w_k, w_{k+1} \}$ of $Q_w$ adjacent to the vertex $w \in Q_w$. In fact, the facet $F_j$ is parallel to the hyperplane spanned by the vectors
    \begin{equation} \label{eq:epsini} 
    \varepsilon_i = w_i - w_{k+1} = \sigma_i(w) - w = - 2 \langle w, n_i \rangle n_i = - 2 \alpha_i V n_i 
    \end{equation}
    for $i \in [k] \setminus \{j\}$, which are all perpendicular to the vector
    $$ q_j = (-1)^{j-1} \varepsilon_1 \times \cdots \times \widehat \varepsilon_j \times \cdots \times \varepsilon_k, $$
    defined as in \eqref{eq:qj}. Note that $p_j$ and $q_j$ are parallel with
    \begin{equation} \label{eq:qjpj}
    q_j = \frac{(-2V)^{k-1}}{\alpha_j} \left( \prod_{\ell \in [k]} \alpha_\ell \right) p_j. 
    \end{equation}
    It follows now from Corollary \ref{cor:altIavnCvD} of our volume ratio formula that
    $$ M_{k+1,i} = - \frac{1}{(k-2)!} \cdot \frac{|\det(\varepsilon_1,\dots,\varepsilon_k)|^{k-2}}
    {\prod_{j \in [k] \setminus \{i\}} |\langle w_{k+1},q_j\rangle|}. $$
    We have with $\alpha = (\alpha_1,\dots,\alpha_k)^\top$,
    \begin{multline*}
        \langle w_{k+1},q_j \rangle = \sum_{i \in [k]} \alpha_i \langle p_i,q_j \rangle \stackrel{\eqref{eq:qjpj}}{=} 
        \frac{(-2V)^{k-1}}{\alpha_j} \left( \prod_{\ell \in [k]} \alpha_\ell \right)
        \sum_{i \in [k]} \alpha_i \langle p_i,p_j \rangle \\ \stackrel{\eqref{eq:pipj}}{=} \frac{(-2)^{k-1}V^{k+1}}{\alpha_j} \left( \prod_{\ell \in {[k]}} \alpha_\ell \right) (N^{-1} \alpha)_j 
    \end{multline*}
    and
    $$ 
    \det(\epsilon_1,\dots,\epsilon_k) \stackrel{\eqref{eq:epsini}}{=} (-2)^k V^k \left( \prod_{\ell \in [k]} \alpha_\ell \right) \det(n_1,\dots,n_k) \stackrel{\eqref{eq:V}}{=} (-2)^k V^{k+1} \left( \prod_{\ell \in [k]} \alpha_\ell \right). 
    $$
    This implies 
    \begin{multline*}
       M_{k+1,i} = - \frac{1}{(k-2)!} \frac{2^{k(k-2)} V^{(k+1)(k-2)} \left( \prod_{\ell \in [k]} \alpha_\ell \right)^{k-2} \left( \prod_{\ell \in [k] \setminus \{i\}} \alpha_\ell \right)}
       {2^{(k-1)^2}V^{k^2-1}\left( \prod_{\ell \in [k]} \alpha_\ell \right)^{k-1} \prod_{j \in [k] \setminus \{i\}} | (N^{-1}\alpha)_j |} \\ = - \frac{1}{2 V^{k+1} (k-2)!\cdot
       \alpha_i \cdot \prod_{j \in [k]\setminus \{i\}} |(N^{-1} \alpha)_j |}.
    \end{multline*}
\end{proof}

Next we investigate the spectral gaps of the stochastic matrices $T(Q_w)$   
of permutahedra $Q_w$ for varying $w \in \mathcal{F}_0$. (Note that the associated random walks are without laziness due to the choice $\tau=\mu$.) 
Since all generators $s_j$ have order $2$, the simplex $\Delta_\Gamma$ of the probability weights in \eqref{eq:simpDelGamm} takes the form
$$ \Delta_\Gamma = \set{x_1,\dots,x_k > 0}{\sum_{j=1}^k x_j = 1}. $$
Then every $X \in \Delta_\Gamma$ is in $1:1$-correspondence to a doubly stochastic matrix $P_X$ with $(k+1)$-st row equal to
$$ (x_1,\cdots,x_k,0,\dots,0) $$
(in accordance with the more general situation \eqref{eq:(m+1)-row-TQ} given earlier). It follows from \eqref{eq:stochrelvertices} in this paper, the fact that the elements $s_1,\dots,s_k$ form a minimal symmetric set of generators of the Coxeter group $\Gamma$, and from formula (6), Corollary 1.7 and Proposition 1.9 in \cite{IP-13} that the restrictions of our maps $\Psi_\Delta$
and $\lambda_1$ in \eqref{eq:PsiDelta} and \eqref{eq:Psilambda} to $\mathcal{F}_0$ agree with the corresponding maps $\Psi_\Delta$ and $\Psi_\lambda$ given in \cite{IP-13}, respectively. Moreover, both maps are smooth on $\mathcal{F}_0$ due to Corollary 1.7 in \cite{IP-13}. Finally, we have the agreement $T(Q_w) = P_X$ for all $w \in \mathcal{F}_0$ and $X = \Psi_\Delta(w)$ with $P_X$ defined in \cite[Section 1.3]{IP-13}.

\begin{theorem}
  Let $\Gamma$ be an irreducible finite Coxeter group of rank $k$ with geometric realisation $\Gamma \hookrightarrow O(k)$ as a Euclidean reflection group, $\sigma_1,\dots,\sigma_k$ and $N$ be the reflections and Schl\"afli matrix associated to the simple roots $n_1,\dots,n_k \in \R^k$ introduced before. Let
  $$ \mathcal{F}_0 = \set{p = \sum_{j=1}^k \alpha_j p_j}{p \in S^{k-1}(1), \alpha_1,\dots,\alpha_k > 0 } $$
  be an open spherical fundamental domain of the $\Gamma$-action with $p_1,\dots,p_k \in \R^k$ defined in \eqref{eq:pj}. Then the map
  $$ \lambda_1: \mathcal{F}_0 \to [-1,1] $$
  has the property $\lambda_1(w) \to 1$ whenever $w \to \partial \mathcal{F}_0$, and $\lambda_1$ assumes a global minimum at a unique point $w_0 \in \mathcal{F}_0$ given by
  $$ w_0 = \frac{\sum_{j=1}^k p_j}{\left\Vert \sum_{j=1}^k p_j \right\Vert}, $$
  that is, the corresponding coefficient vector $\alpha_0 = (\alpha_1,\dots,\alpha_k)$ is a constant vector. Moreover, the corresponding permutahedron $Q_{w_0}$ is equilateral (all its edges have the same length) and there is no other permutahedron $Q_w$ with $w \in \mathcal{F}_0$ which is equilateral.
\end{theorem}

\begin{proof}
  Since $s_1,\dots,s_k$ is a minimal symmetric set of generators of the Coxeter group $\Gamma$, Proposition 1.2 in \cite{IP-13} implies that
  $$ \lambda_1(w) \to 1 \quad \text{as $w \to \partial \mathcal{F}_0$}, $$
  and that a global minimum of $\lambda_1$ is assumed inside the open set $\mathcal{F}_0 \subset S^{k-1}(1)$. 

  Izmestiev's constant corank result implies that the multiplicity of $\lambda_1(Q_w)$ is constant and equal to $k$ for all $w \in \mathcal{F}_0$. Therefore, the conditions in Theorem 1.3 in \cite{IP-13} are satisfied and any critical point $w$ of $\lambda_1: \mathcal{F}_0 \to [0,1]$ corresponds to equilaterality of $Q_w$. Using the same notation $w = w_{k+1}$ and $q_j = \sigma_j(w)$ for $j \in [k]$ as in Corollary \ref{cor:CoxMdescr}, vertex transitivity of $Q_w$ implies that equilaterality of $Q_w$ holds precisely if $\Vert w - w_i \Vert$ is constant for all $i \in [k]$. Since
  $$ \Vert w - w_i \Vert = 2 \Vert \langle w,n_i \rangle n_i \Vert = |\langle w,n_i \rangle| $$
  and
  $$ \langle w, n_i \rangle = \sum_j \alpha_j \langle p_j,n_i \rangle = \alpha_i V >0,$$
  equalaterality of $Q_w$ holds precisely if the vector $\alpha_0 = (\alpha_1,\dots,\alpha_k)$ is constant. This shows that $\lambda_1$ has a unique minimum inside $\mathcal{F}_0$ at the point $w_0\in \mathcal{F}_0$ given in the theorem.
\end{proof}

In this example, $S_{\rm{reg}}^{k-1}$ is the disjoint union of $\Gamma$-images of the spherical fundamental domain $\mathcal{F}_0$ in Corollary \ref{cor:CoxMdescr} and the number of connected components of $S_{\rm{reg}}^{k-1}$ agrees with the cardinality of $\Gamma$. Moroever, only the convex hulls $Q_w$ of orbits of points $w \in S_{\rm{reg}}^{k-1}$ are called permutahedra of the Coxeter group, they have $|\Gamma|$ many vertices and are all of the same combinatorial type. In contrast, the convex hull $Q_w$ of an orbit of a point $w \in S^{k-1} \setminus S_{\rm{reg}}^{k_1}$ has a strictly smaller number of vertices. The combinatorial types of the polytopes $Q_w$ for the Coxeter groups $A_3, B_3$ and $H_3$ are illustrated in Figures 2 and 4 of \cite{IP-13}, where the polytopes corresponding to $w \in S^2 \setminus S_{\rm{reg}}^2$ have still the combinatorial structures of certain Archimedean solids. We like to emphasize that this example differs significantly from the following one. There, in Subsection \ref{subsec:rottet}, $S_{\rm{reg}}^2$ is a punctured sphere consisting therefore of only of one connected component. Moreover, in contrast to permutahedra of Coxeter groups, the combinatorial type of $Q_w$ changes within the regular set $S_{\rm{reg}}^2$. 
  
\subsection{Example: The pure rotational tetrahedral group}
\label{subsec:rottet}

Let $\Gamma \subset SO(3)$ be the group
$$ \Gamma = \set{ \begin{pmatrix}\epsilon_1 & 0 & 0 \\ 0 & \epsilon_2 & 0 \\ 0 & 0 & \epsilon_3 \end{pmatrix}, 
\begin{pmatrix} 0 & \epsilon_1 & 0 \\ 0 &  0 & \epsilon_2 \\ \epsilon_3 & 0 & 0 \end{pmatrix},
\begin{pmatrix} 0 & 0 & \epsilon_1 \\ \epsilon_2 & 0 & 0 \\ 0 & \epsilon_3 & 0 \end{pmatrix}}{\epsilon_1,\epsilon_2,\epsilon_3 \in \{-1,1\}, \epsilon_1 \epsilon_2 \epsilon_3 = 1}, $$
which has $12$ elements.
Since $\Gamma$ acts transitively on the vertices on the regular tetrahedron 
$$ Q = \conv\{v_1,v_2,v_3,v_4\} \subset \mathbb{R}^3$$
with
$$ v_1 = (1,1,1)^\top, \quad v_2 = (1,-1,-1)^\top, \quad v_3 =  (-1,-1,1)^\top, \quad v_4 = (-1,1,-1)^\top, $$
by rotations, it is called the 
\emph{pure rotational tetrahedral group}. Various properties of this group are discussed in \cite{XC-19}. The map
$$ S^2(r) \ni w \mapsto Q_w = \conv(\Gamma w) $$
is given by
$$ Q_{(a,b,c)} = \conv\set{(\epsilon_1 a,\epsilon_2 b,\epsilon_3 c)^\top,(\epsilon_1 b,\epsilon_2 c,\epsilon_3 a)^\top,(\epsilon_1 c,\epsilon_2 a,\epsilon_3 b)^\top}{\epsilon_1,\epsilon_2,\epsilon_3 \in \{-1,1\}, \epsilon_1 \epsilon_2 \epsilon_3 = 1}. $$
The regular subset of the $\Gamma$-action on $S^2_{\rm{reg}}(r)$ is given by
$$ S_{\rm{reg}}^{2}(r) = S^2(r) \setminus \left\{ (\pm \frac{r}{\sqrt{3}},\pm \frac{r}{\sqrt{3}},\pm \frac{r}{\sqrt{3}})^\top, (\pm r,0,0)^\top, (0,\pm r,0)^\top, (0,0,\pm r)^\top \right\},  $$
that is, $S_{\rm{reg}}^2(r)$ is a $14$-punctured sphere.

By applying suitable elements of $\Gamma$ to a generic orbit representative $(a,b,c)^\top \in S^2(r)$ with $|a|,|b|,|c| > 0$ pairwise different, one can see that an open spherical fundamental domain of the $\Gamma$-action on $S^2(r)$ is given by
$$ \mathcal{F}_0 := \set{(a,b,c)^\top \in S^2(r)}{a > b > |c| > 0 \quad \text{or} \quad a > c> |b| > 0}. $$
To cover all different convex hulls of $\Gamma$-orbits, we take the closure
$$ \overline{\mathcal{F}_0} := \set{(a,b,c)^\top \in S^2(r)}{a \ge b \ge |c| \ge 0 \quad \text{or} \quad a \ge c\ge |b| \ge 0}, $$
and restrict our considerations to its subset
$$ \mathcal{M} = \set{(a,b,c)^\top \in S^2(r)}{ a \ge b \ge |c| \ge 0}, $$
which we call the \emph{moduli space} of (convex hulls of) $\Gamma$-orbits in $S^2(r)$. Note that the set $\overline{\mathcal{F}_0} \setminus \mathcal{M}$ is obtained from $\mathcal{M}$ by the isometry permuting the $y$- and $z$-coordinates which does not lead to essentially new polytopes. 
To avoid an excessive use of square-roots, we relax the choice of representatives in $\mathcal{M}$ by no longer requiring that they have a fixed length $r > 0$. Note that, under this relaxation, the regular tetrahedron appears in this variation as the polytope $Q_{(1,1,1)^\top}$. The moduli space $\mathcal{M}$ of the pure rotational tetrahedral group is illustrated in Figure \ref{fig:modspacetetrahed}. The reflective symmetry of the polytopes represented by points in $\mathcal{M}$ in the $c$-coordinate are a consequence of the fact that the points $(a,b,c)^\top \in \mathcal{M}$ and $(a,b,-c)^\top \in \mathcal{M}$ generate essentially the same polytopes via the ambient isometry $(x,y,z) \mapsto (x,y,-z)$ which is, however, not an element in $\Gamma$.  

\begin{figure}[h]
\begin{center}
    \begin{tikzpicture}[scale=0.8]
    \draw[black, thick] [-stealth] (6,0) -- (-6,0);
    \node[align=right] at (-6.5,0) {$c$};
    \draw[black, thick] [-stealth] (-7,1) -- (-7,7);
    \node[align=left] at (-7,7.5) {$b$};
    \draw[black, thick] (-5,-0.2) -- (-5,0.2);
    \node[align=left] at (-5,-0.4) {$a$};
    \draw[black, thick] (0,-0.2) -- (0,0.2);
    \node[align=left] at (0,-0.4) {$0$};
    \draw[black, thick] (5,-0.2) -- (5,0.2);
    \node[align=left] at (5,-0.4) {$-a$};
    \draw[blue] (0,1) -- (-5,6);
    \draw[blue, dashed] (0,1) -- (0,6);
    \draw[blue] (0,1) -- (5,6);
    \draw[blue] (-5,6) -- (5,6);
    \node[blue,align=right] at (-1.8,4.5) {\LARGE$\mathcal{M}_{\rm{reg}}$};
    \draw[black, thick] (-7.2,1) -- (-6.8,1);
    \node[align=left] at (-7.4,1) {$0$};
    \draw[black, thick] (-7.2,6) -- (-6.8,6);
    \node[align=left] at (-7.4,6) {$a$};
    \node[blue,align=right] at (-2.5,6.3) {$a=b$};
    \filldraw[black] (0,1) circle (2pt);
    \node[align=left] at (1.5,0.8) {Octahedron\\ $(1,0,0)$};
    \filldraw[blue] (0,6) circle (2pt);  
    \node[align=center,blue] at (0,6.6) {Cuboctahedron\\ $(1,1,0)$};
    \filldraw[black] (5,6) circle (2pt);
    \node[align=center] at (5,6.6)
    {Tetrahedron\\ $(1,1,-1)$};
    \filldraw[black] (-5,6) circle (2pt);
    \node[align=center] at (-5,6.6)
    {Tetrahedron\\ $(1,1,1)$};
    \filldraw[blue] (-1.6,2.6) circle (2pt);
    \node[align=right,blue] at (-4,2.3)
    {Truncated Tetrahedron\\ $(3,1,1)$};
    \filldraw[blue] (1.6,2.6) circle (2pt);
    \node[align=left,blue] at (4.1,2.3)
    {Truncated Tetrahedron\\ $(3,1,-1)$};
    \node[blue,align=right] at (-4.5,4.5)
    {$b=c$};
    \node[blue,align=left] at (4.5,4.5)
    {$b=-c$};
    \node[blue,align=left] at (0.7,5.3) {$c=0$};
    \filldraw[blue] (0,4) circle (2pt);
    \node[align=left,blue] at (1.3,4)
    {Icosahedron\\ $(\phi,1,0)$};
\end{tikzpicture}
\end{center}
\caption{The moduli space $\mathcal{M}$ of vertex transitive $k$-polytopes $Q_{(a,b,c)^\top}=\conv(\Gamma (a,b,c)^\top)$ with respect to the pure rotational tetrahedral group $\Gamma$ with $\phi = \frac{1 +\sqrt{5}}{2}$ the golden ratio -- the subspace $\mathcal{M}_{\rm{reg}}$ is illustrated in blue. Note that the restriction $(a,b,c)^\top \in S^2(r)$ is dropped, for simplicity. \label{fig:modspacetetrahed}}
\end{figure}
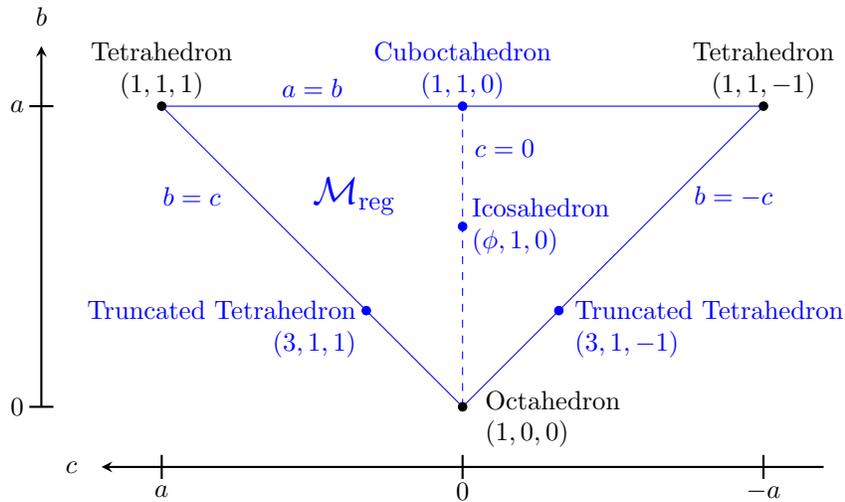

\FloatBarrier

The regular subset 
$$ \mathcal{M}_{\rm{reg}} = \mathcal{M} \setminus \set{(a,b,c)^\top}{a=b=|c|>0 \quad \text{or} \quad a>b=c=0} $$ 
of the moduli space $\mathcal{M}$ corresponding to $S^2_{\rm{reg}}(r)$ is illustrated in Figure \ref{fig:modspacetetrahed} in blue and decomposes into subsets of polytopes with fixed combinatorial type as follows: 


\begin{itemize}
    \item ${\mathcal{M}}_{\rm{icosa}}$ is the generic combinatorial type (see Figure \ref{fig:poly-icosahedral} for a polyhedron falling into this class) whose vertex-edge graph corresponds to that of an Icosahedron:
$$ {\mathcal{M}}_{\rm{icosa}} = \set{(a,b,c)^\top \in \mathcal{M}}{a > b > |c| \ge 0}. $$
In fact, $Q_{(\phi,1,0)^\top}$ with the golden ratio $\phi = \frac{1+\sqrt{5}}{2}$ is a regular icosahedron and therefore \emph{equilateral}.
\item $\mathcal{M}_{\rm{trunctetra}}^\pm$ is the combinatorial type whose vertex-edge graph corresponds to that of a truncated tetrahedron:
$$ \mathcal{M}_{\rm{trunctetra}}^\pm = \set{(a,b,c)^\top \in \mathcal{M}}{a > b = \pm c > 0}. $$
In fact, $Q_{(3,1,\pm 1)^\top}$ are \emph{equilateral} truncated tetrahedra.
\item $\mathcal{M}_{\rm{cubocta}}$ is the combinatorial type whose vertex-edge graph corresponds to that of a cuboctahedron:
$$ \mathcal{M}_{\rm{cubocta}} = \set{(a,b,c)^\top \in \mathcal{M}}{a = b > |c| \ge 0}. $$
In fact, $Q_{(1,1,0)^\top}$ is an \emph{equilateral}
cuboctahedron.
\end{itemize}
Note that $\mathcal{M} \setminus \mathcal{M}_{\rm{reg}}$ decomposes similarly into the discrete subsets $\mathcal{M}_{\rm{tetra}}^\pm = \{(1,1,\pm 1) \}$ and $\mathcal{M}_{\rm{octa}} = \{ (1,0,0) \}$, corresponding to regular tetrahedra and regular octahedron, respectively, which are \emph{equilateral} polyhedra. 

Let us now focus on the subset $\mathcal{M}_{\rm{icosa}} \subset \mathcal{M}_{\rm{reg}}$ corresponding to polytopes of icosahedral type. The vertex $w = (a,b,c)^\top \in \mathcal{M}_{\rm{icosa}}$ has five neighbours, given by
$$ \gamma_1 w = (a,-b,-c)^\top, \,  \gamma_2 w = (b,c,a)^\top, \, \gamma_2^{-1} w = (c,a,b)^\top, \, \gamma_3 w = (b,-c,-a)^\top, \, \gamma_3^{-1} w = (-c,a,-b)^\top $$
with $\gamma_1 = \gamma_1^{-1}$, and we enumerate them by $w_1,w_2,w_3,w_4,w_5 \in Q_w$, respectively. Here, $\gamma_1,\gamma_2$ and $\gamma_3$ are generators of $\Gamma$, given by 
\begin{equation} \label{eq:gamma123} 
\gamma_1 = \begin{pmatrix} 1 & 0 & 0 \\ 0 & -1 & 0 \\ 0 & 0 & -1 \end{pmatrix}, \quad \quad \gamma_2 = \begin{pmatrix} 0 & 1 & 0 \\ 0 & 0 & 1 \\ 1 & 0 & 0 \end{pmatrix}, \quad \quad \gamma_3 = \begin{pmatrix} 0 & 1 & 0 \\ 0 & 0 & -1 \\ -1 & 0 & 0 \end{pmatrix}. 
\end{equation}
Note that $\gamma_1 = \gamma_1^{-1}$.
Let $w_6 = w = (a,b,c) \in \mathcal{M}_{\rm{icosa}}$ and $w_7,\dots,w_{12}$ be the following remaining vertices: of $Q_w$:
\begin{align*}
    w_7 &= (-a,b,-c)^\top, & w_8 &= (-b,c,-a)^\top, & w_9 &= (c,-a,-b)^\top, \\
    w_{10} &= (-a,-b,c)^\top, & w_{11} &= (-b,-c,a)^\top, & w_{12} &= (-c,-a,b)^\top.
\end{align*}
The polyhedron $Q_w$ (for the special choice $w=(3,2,1)^\top$) is illustrated in Figure \ref{fig:poly-icosahedral} and has $20$ triangular faces given by
{\small{\begin{align*}
    f_1 &= \conv(w_1,w_2,w_6), & f_2 &= \conv(w_1,w_2,w_{12}), & f_3 &= \conv(w_1,w_4,w_6), & f_4 &= \conv(w_1,w_4,w_9), \\
    f_5 &= \conv(w_1,w_9,w_12), & f_6 &= \conv(w_2,w_3,w_6), & f_7 &= \conv(w_2,w_3,w_{11}), & f_8 &=\conv(w_2,w_{11},w_{12}), \\
    f_9 &= \conv(w_3,w_5,w_6), & f_{10} &= \conv(w_3,w_5,w_7), & f_{11} &= \conv(w_3,w_7,w_{11}), & f_{12} &=\conv(w_4,w_5,w_6), \\
    f_{13} &= \conv(w_4,w_5,w_8), & f_{14} &= \conv(w_4,w_8,w_9), & f_{15} &= \conv(w_5,w_7,w_8), & f_{16} &=\conv(w_7,w_8,w_{10}), \\
    f_{17} &= \conv(w_7,w_{10},w_{11}), & f_{18} &= \conv(w_8,w_9,w_{10}), & f_{19} &= \conv(w_9,w_{10}8,w_{12}), & f_{20} &=\conv(w_{10},w_{11},w_{12}).
\end{align*}
}}

\begin{figure}
     \begin{center}
        \includegraphics[width=0.44\textwidth]{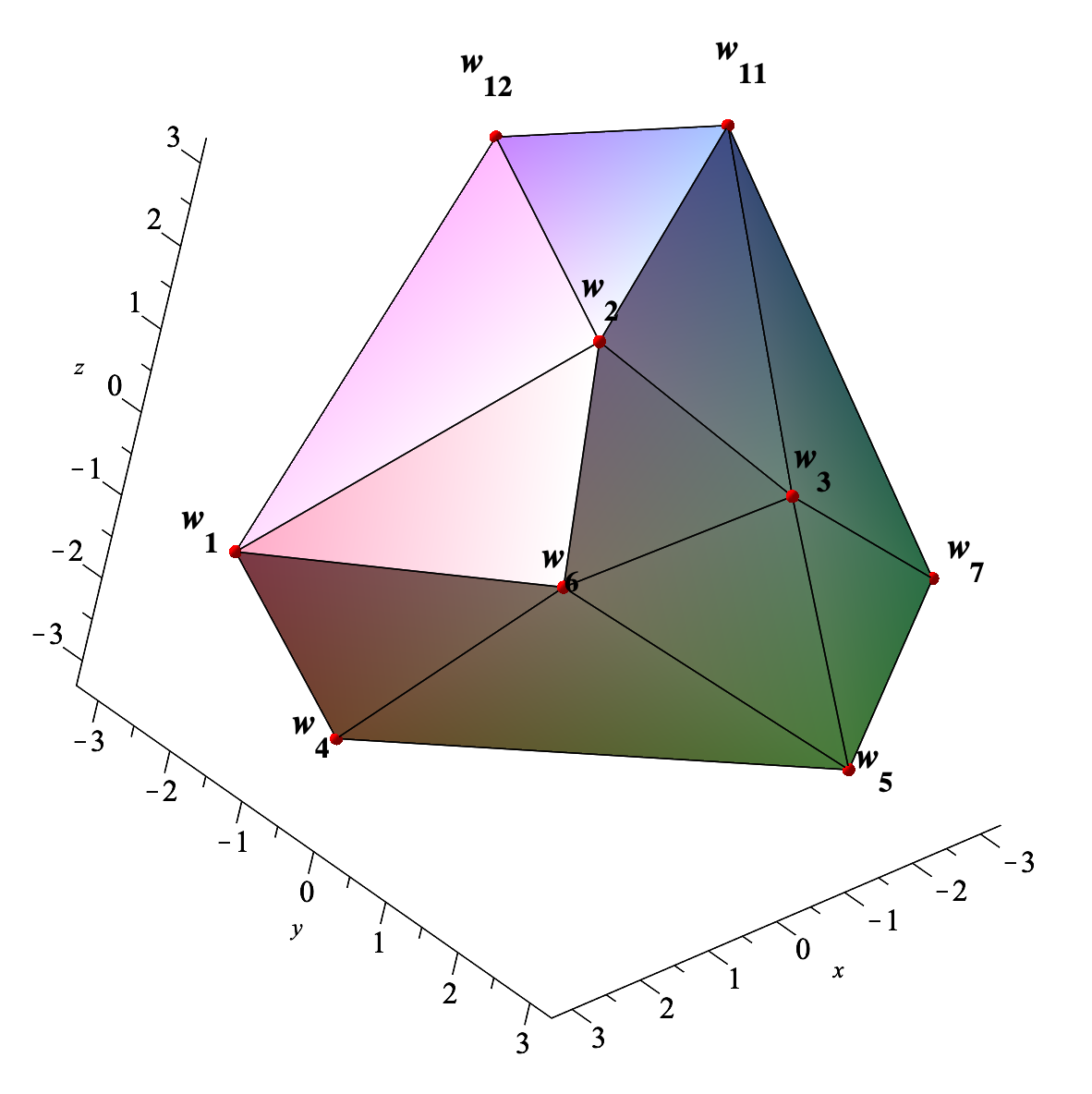}
        \includegraphics[width=0.54\textwidth]{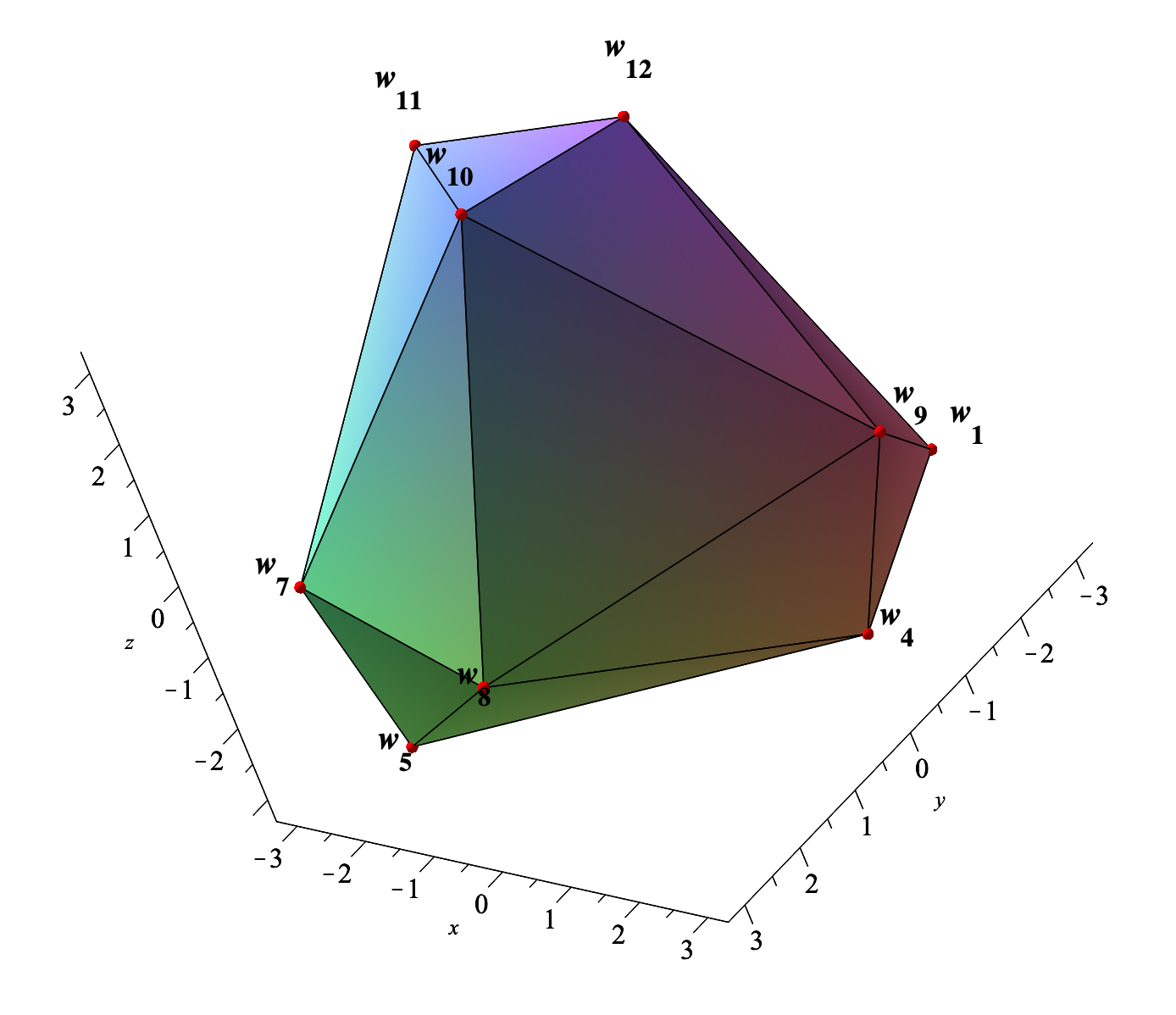}
     \end{center}
     \caption{Different views of the polytope $Q_w$ of icosahedral type with $w = (3,2,1)^\top$ 
     \label{fig:poly-icosahedral}}
\end{figure}

\FloatBarrier

The faces $f_2, f_4, f_6, f_{11}, f_{12}, f_{15}, f_{18}, f_{20}$ are equilateral triangles for any choice of $w_6 = w \in \mathcal{M}_{\rm{icosa}}$.

\smallskip

Let us now derive the weights $x_1,x_2,x_3$ of the random walk matrix $T(Q_w)$. For any $w = w_6 \in \mathcal{M}$, the equation
$$ w_6 = x_1' w_1 + x_2' (w_2+w_3) + x_3' (w_4+w_5) $$
has a unique solution, given by
$$ 
x_1' = \frac{(a-b)(a+b+c)(a+b-c)}{D'}, \quad
x_2' = \frac{a(a+b-c)(b+c)}{D'}, \quad
x_3' = \frac{a(a+b+c)(b-c)}{D'},
$$
with $D' = a^3 + a^2 b + a b^2 +b^3 + a c ^2 - b c^2 > 0$. Choosing $(x_1,x_2,x_3) = \mu (x_1',x_2',x_3')$ such that $x_1+2(x_2+x_3) = 1$
leads to
\begin{equation} \label{eq:abcx1x2x3} 
x_1 = \frac{(a-b)(a+b+c)(a+b-c)}{D}, \quad
x_2 = \frac{a(a+b-c)(b+c)}{D}, \quad
x_3 = \frac{a(a+b+c)(b-c)}{D},
\end{equation}
with $D = a^3+5a^2 b +3ab^2 - 5ac^2 - b^3 + bc^2 > 0$. 
Note that this transformation from $w = (a,b,c)^\top \in \mathbb{R}^3 \neq \{\0\}$ with $a \ge b \ge |c|$  into $(x_1,x_2,x_3)$ is not only well-defined on $\mathcal{M}_{\rm{isoca}}$ but on all of $\mathcal{M}_{\rm{reg}}$, and its outcome is invariant under rescaling $w \mapsto \mu w$, $\mu>0$. The $\Gamma$-action is simply transitive on the $12$ vertices of any polytope $Q_w$ with $w \in \mathcal{M}_{\rm{reg}}$, and we have $x_2=0$ or $x_3=0$ for all polytopes represented by $\mathcal{M}_{\rm{trunctetra}}^\pm$ and $x_1=0$ for all polytopes represented by $\mathcal{M}_{\rm{cubocta}}$. Table \ref{tab:values-tetrahedral-group} provides information about this transformation at specific points. 

The random walk matrix $T(Q_w)$, in terms of the weights $(x_1,x_2,x_3)$, takes the form
\begin{equation} \label{eq:TQw}
T(Q_w) = \begin{pmatrix}
0   & x_3 & 0   & x_2 & 0   & x_1 & 0   & 0   & x_2 & 0   & 0   & x_3 \\  
x_3 & 0   & x_2 & 0   & 0   & x_2 & 0   & 0   & 0   & 0   & x_1 & x_3 \\
0   & x_2 & 0   & 0   & x_1 & x_2 & x_3 & 0   & 0   & 0   & x_3 & 0 \\
x_2 & 0   & 0   & 0   & x_3 & x_3 & 0   & x_1 & x_2 & 0   & 0   & 0 \\
0   & 0   & x_1 & x_3 & 0   & x_3 & x_2 & x_2 & 0   & 0   & 0   & 0 \\
x_1 & x_2 & x_2 & x_3 & x_3 & 0   & 0   & 0   & 0   & 0   & 0   & 0 \\
0   & 0   & x_3 & 0   & x_2 & 0   & 0   & x_2 & 0   & x_1 & x_3 & 0 \\
0   & 0   & 0   & x_1 & x_2 & 0   & x_2 & 0   & x_3 & x_3 & 0   & 0 \\
x_2 & 0   & 0   & x_2 & 0   & 0   & 0   & x_3 & 0   & x_3 & 0   & x_1 \\
0   & 0   & 0   & 0   & 0   & 0   & x_1 & x_3 & x_3 & 0   & x_2 & x_2 \\
0   & x_1 & x_3 & 0   & 0   & 0   & x_3 & 0   & 0   & x_2 & 0   & x_2 \\
x_3 & x_3 & 0   & 0   & 0   & 0   & 0   & 0   & x_1 & x_2 & x_2 & 0
\end{pmatrix},
\end{equation}
and its characteristic polynomial is given by
\begin{multline*} 
\det(T(Q_w) - x \cdot \ID_{12}) =
(x-1)\cdot (x-x_1+x_2+x_3)^2 \cdot \\
(x^3+x_1 \,x^2 + (-x_1^2-3x_2^2+2x_2x_3-3x_3^2)\,x -x_1^3-x_1 x_2^2-2x_1x_2x_3-x_1x_3^2 - 2x_2^3+2x_2^2x_3+2x_2x_3^2-2x_3^3)^3.
\end{multline*}
Therefore the stochastic matrix $T(q_w)$ has the following spectrum for $w \in \mathcal{M}_{\rm{reg}}$: a simple eigenvalue $\lambda_0(w)=1$, three eigenvalues $\lambda_1(w) > \lambda_2(w) > \lambda_3(w)$ of multiplicity $3$ which are the roots of the third polynomial factor on the right hand side, and the eigenvalue $\lambda_4(w) = x_1 - (x_2+x_3)$ with $x_1,x_2,x_3$ given by \eqref{eq:abcx1x2x3} of multiplicity $2$. The eigenvalue functions over the moduli space are illustrated in Figure \ref{fig:eigenvalues}. We chose a slightly different perspective than the one given in Figure \ref{fig:modspacetetrahed} for a better illustration of these functions. Note that $\lambda_1(w) > \lambda_4(w) > \lambda_3(w)$, but $\lambda_2$ and $\lambda_4$ have not ordering in size.  

\begin{figure}[h]
     \begin{center}
        \includegraphics[width=0.48\textwidth]{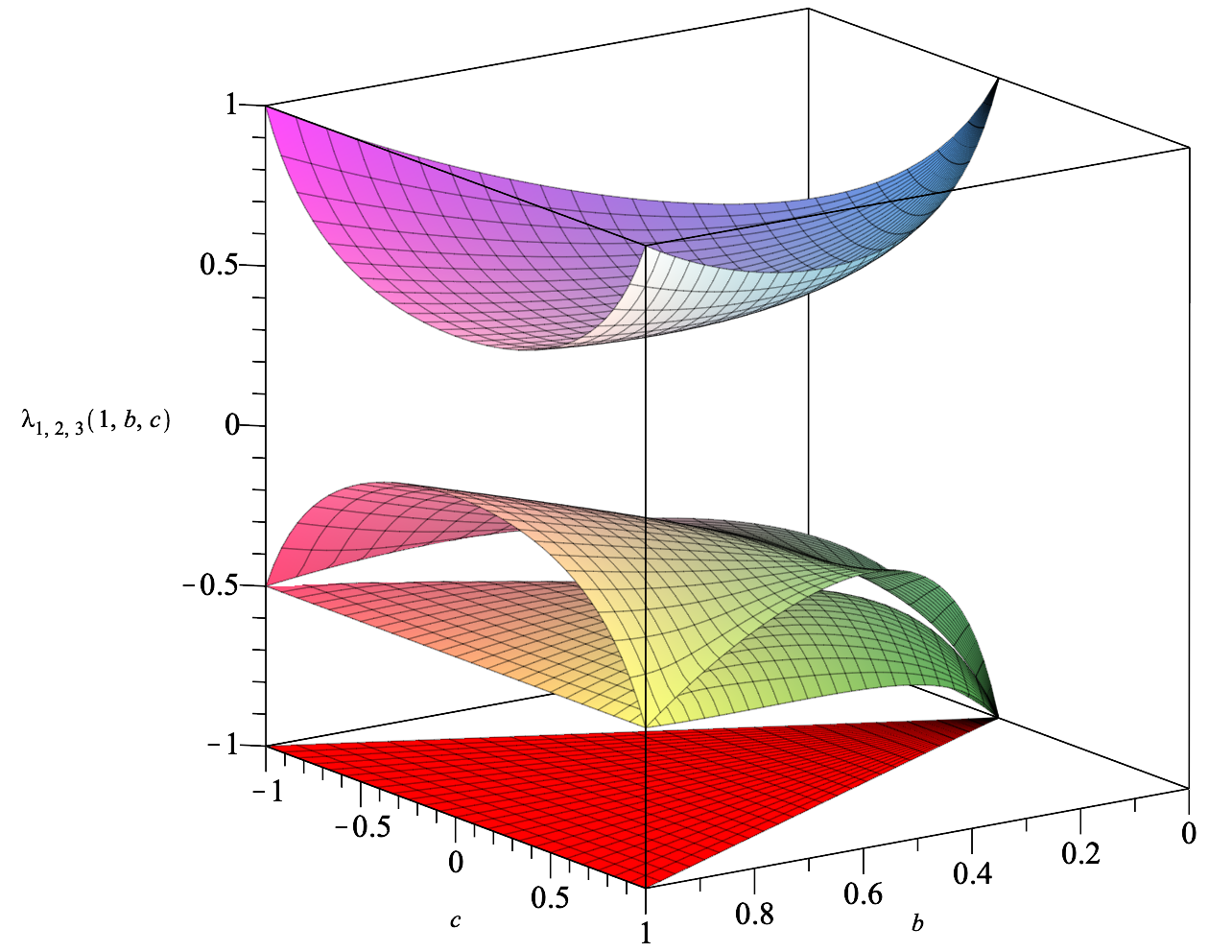}
        \includegraphics[width=0.48\textwidth]{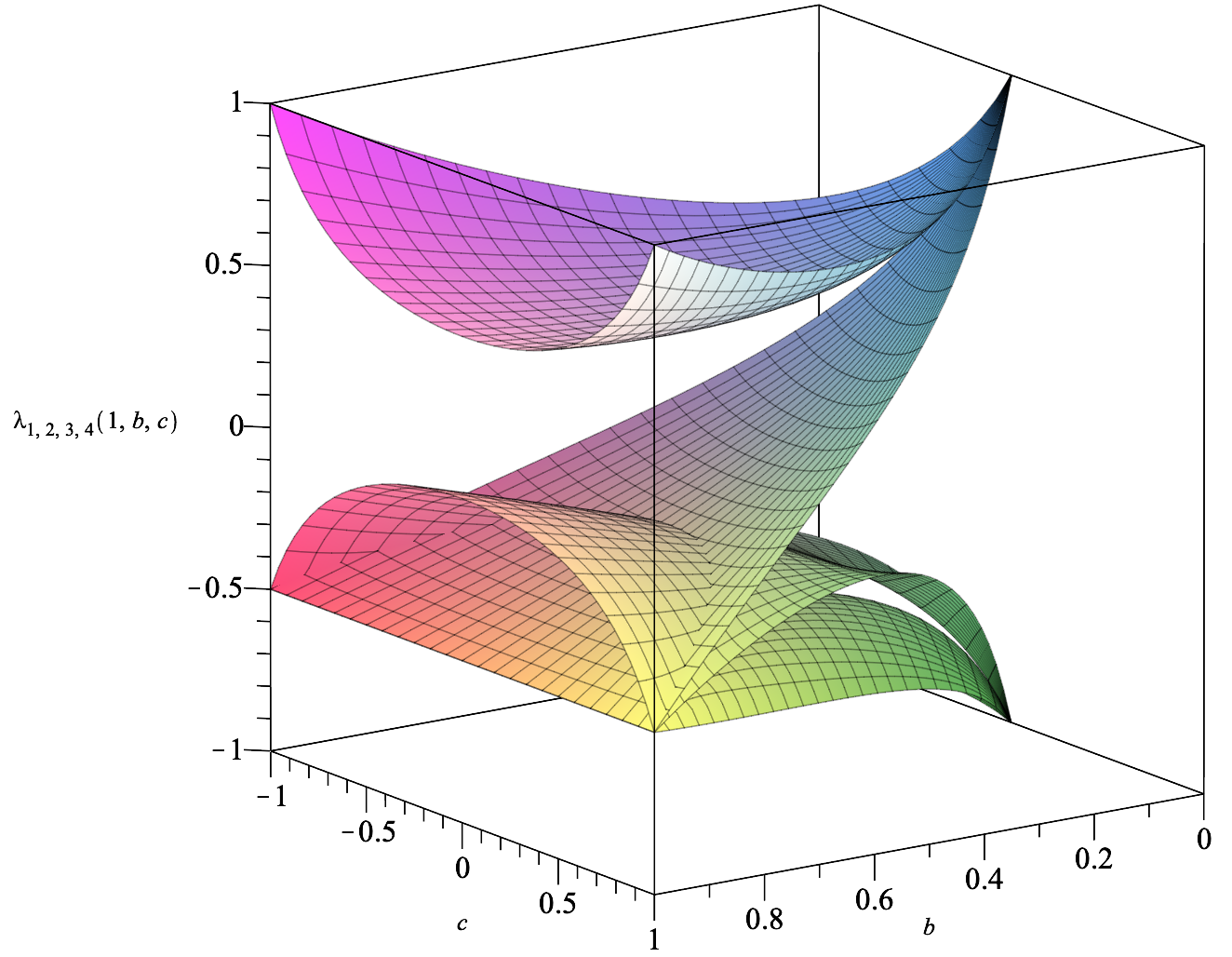}
     \end{center}
     \caption{The functions $\lambda_1,\lambda_2,\lambda_3$ corresponding to the eigenvalues of multiplicity $3$ of $T(Q_w)$ (left) and to all eigenvalues except the trivial one $\lambda_0 \equiv 1$ (right) over the regular moduli space $\mathcal{M}_{\rm{reg}}$. The red triangle at the bottom of the left hand figure illustrates the triangular domain $\mathcal{M}_{\rm{reg}}$ of the eigenfunctions $\lambda_i(a,b,c)$ with fixed $a=1$.
     \label{fig:eigenvalues}}
\end{figure}

\FloatBarrier

Let us now focus on the second largest eigenvalue $\lambda_1(w)$ of $T(Q_w)$. Figure \ref{fig:lambda1tetrahed} illustrates this function over the moduli space $\mathcal{M}_{\rm{reg}}$ and Table \ref{tab:values-tetrahedral-group} provides concrete $\lambda_1$-values at specific points.
The points of $\mathcal{M}_{\rm{reg}}$ corresponding to the \emph{equilateral} polyhedra are marked by vertical black lines, and we see that the $\lambda_1$-function over $\mathcal{M}_{\rm{reg}}$ has a global minimum at the equilateral icosahedron. We noticed also in Figure \ref{fig:eigenvalues} (left) that the global minimum of $\lambda_1$ agrees with the global maximum of $\lambda_3$. Moreover, $\lambda_1$ has local minima within the subsets $\mathcal{M}_{\rm{cubocta}}$ and $\mathcal{M}_{\rm{trunctetra}}^\pm$ at the equilateral cuboctahedron and the equilateral truncated tetrahedra, respectively. The $\Gamma$-action is no longer simply transitive on 
the vertices of the tetrahedra $\conv(\Gamma (1,1,\pm 1)^\top)$ and the octahedron $\conv(\Gamma (1,0,0)^\top)$, for which the random walk matrices have smaller sizes and assume the form
$$ T(Q_w) = \begin{pmatrix} 0 & \frac{1}{3} & \frac{1}{3} & \frac{1}{3} \\
\frac{1}{3} & 0 & \frac{1}{3} & \frac{1}{3} \\ \frac{1}{3} & \frac{1}{3} & 0 & \frac{1}{3} \\ \frac{1}{3} & \frac{1}{3} & \frac{1}{3} & 0 \end{pmatrix} \quad \text{and} \quad T(Q_w) = \begin{pmatrix} 0 & \frac{1}{4} & 0 & \frac{1}{4} & \frac{1}{4} & \frac{1}{4} \\ \frac{1}{4} & 0 & \frac{1}{4} & 0 & \frac{1}{4} & \frac{1}{4} \\ 0 & \frac{1}{4} & 0 & \frac{1}{4} & \frac{1}{4} & \frac{1}{4} \\
\frac{1}{4} & 0 & \frac{1}{4} & 0 & \frac{1}{4} & \frac{1}{4} \\ \frac{1}{4} & \frac{1}{4} & \frac{1}{4} & \frac{1}{4} & 0 & 0 \\ \frac{1}{4} & \frac{1}{4} & \frac{1}{4} & \frac{1}{4} & 0 & 0 \end{pmatrix},$$
with corresponding $\lambda_1$-values $-\frac{1}{3}$ and $-\frac{1}{2}$, respectively. We note that at these points the $\lambda_1$-function is discontinuous. 

\begin{figure}[h]
    \begin{center}
        \includegraphics[width=11cm]{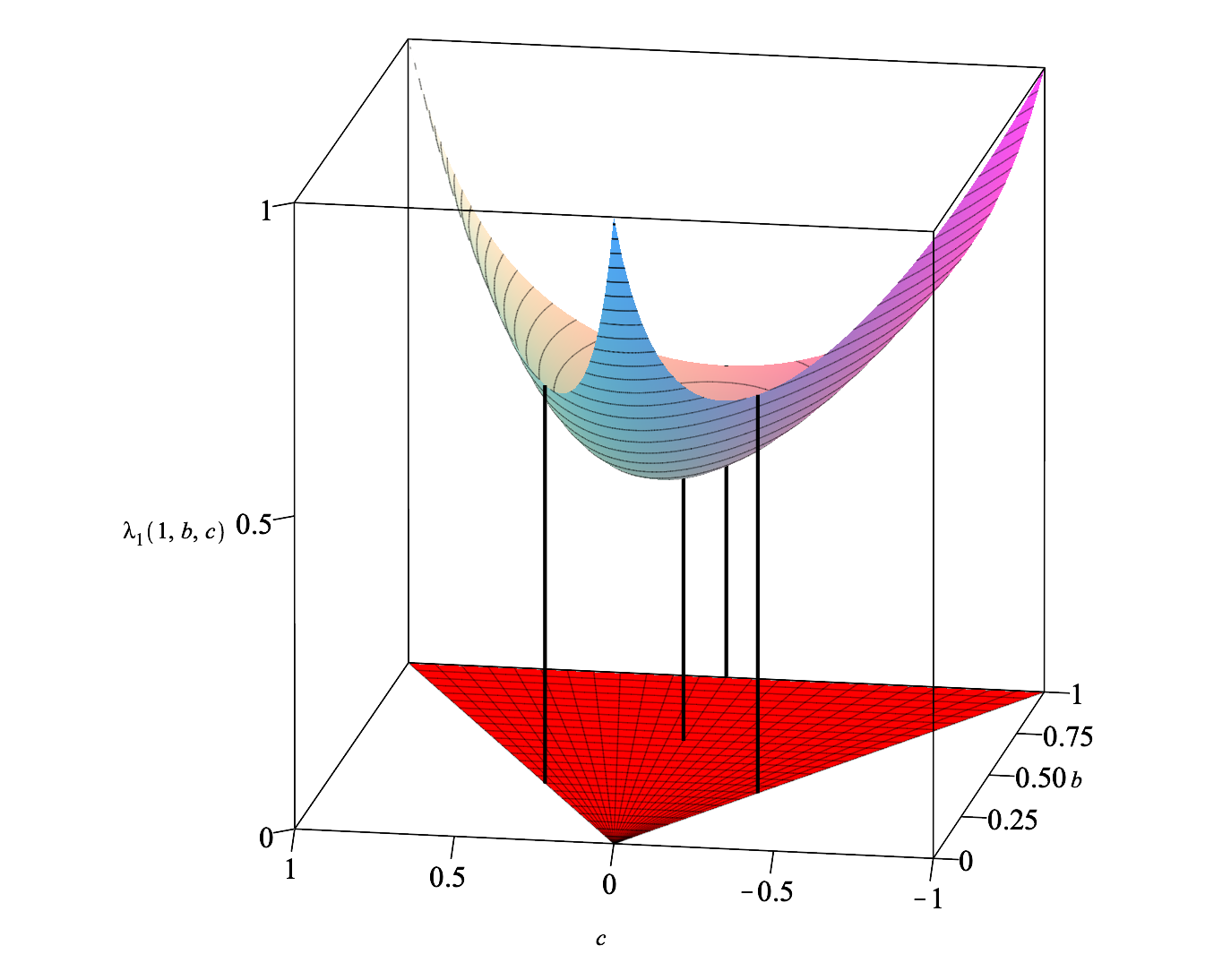}
    \end{center}
    \caption{The function $w \mapsto \lambda_1(w)$ over the moduli space $\mathcal{M}_{\rm{reg}}$; note that the restriction $(a,b,c) \in S^2(r)$ is dropped, and points $w=(1,b,c)^\top$, $1 > b > |c|$ are used, instead.
    \label{fig:lambda1tetrahed}}
\end{figure}



\begin{table}
    \begin{tabular}{l | r r r|l l l|l}
    Polytope & $a$ & $b$ & $c$ & $x_1$ & $x_2$
 & $x_3$ & $\lambda_1$ \\
    \hline
    Tetrahedron & $1$ & $1$ & $-1$ & $0$ & $0$ & $\frac{1}{2}$ & $-\frac{1}{3}$ \color{blue}{($1$)}\\[2pt]
    Cuboctahedron & $1$ & $1$ & $0$ & $0$ & $\frac{1}{4}$ & $\frac{1}{4}$ & $\frac{1}{2}$ \\[2pt]
    Tetrahedron & $1$ & $1$ & $1$ & $0$ & $\frac{1}{2}$ & $0$ & $-\frac{1}{3}$ \color{blue}{($1$)} \\[2pt]
    Icosahedron & $\phi$ & $1$ & $0$ & $\frac{1}{5}$ & $\frac{1}{5}$ & $\frac{1}{5}$ & $0.447214\dots$ \\[2pt]
    Truncated Tetrahedron & $3$ & $1$ & $-1$ & $\frac{5}{11}$ & $0$ & $\frac{3}{11}$ & $\frac{7}{11}$ \\[2pt]
    Truncated Tetrahedron & $3$ & $1$ & $1$ & $\frac{5}{11}$ & $\frac{3}{11}$ & $0$ & $\frac{7}{11}$ \\[2pt]
    Octahedron & $1$ & $0$ & $0$ & $1$ & $0$ & $0$ & $-\frac{1}{2}$ \color{blue}{($1$)}
    \end{tabular}
    \caption{Transformation of $w=(a,b,c)^\top$ values in $(x_1,x_2,x_3)$-weights and second largest eigenvalue $\lambda_1(w)$ of $T(Q_w)$; the $\lambda_1$-values in blue are continuous extensions of the $\lambda_1$-function on $\mathcal{M}_{\rm{reg}}$ to these extremal points of the moduli space. \label{tab:values-tetrahedral-group}}
\end{table}

In the case 
$(a,b,c)=(2,2,1)$, that is, the type of a cuboctahedron (see left image in Figure \ref{fig:trunctetraandcuboct}),
we have $14$ faces, the $8$ equilateral triangles $f_2,f_4,f_6,f_{11},f_{12},f_{15},f_{18},f_{20}$ and the six rectangles 
\begin{align*}
    r_1 &= \conv(w_1,w_2,w_6,w_4), & r_2 &= \conv(w_1,w_9,w_{10},w_{12}), & r_3 &= \conv(w_2,w_3,w_{11},w_{12}),  \\
    r_4 &= \conv(w_3,w_6,w_5,w_7), & r_5 &= \conv(w_4,w_5,w_8,w_9), &
    r_6 &= \conv(w_7,w_8,w_{10},w_{11}).
\end{align*} 
 
In the case $(a,b,c)=(3,1,1)$, 
that is, the type of a truncated tetrahedron (see right image in Figure \ref{fig:trunctetraandcuboct}),
we have $8$ faces, the four equilateral triangles $f_4,f_6,f_{15},f_{20}$ and four hexagons given by
\begin{align*}
    h_1 &= \conv(w_1,w_6,w_2,w_{11},w_{12},w_9), & h_2 &= \conv(w_1,w_4,w_8,w_5,w_3,w_6),  \\
    h_3 &= \conv(w_2,w_3,w_5,w_7,w_{10},w_{11}), & h_4 &= \conv(w_4,w_8,w_7,w_{10},w_{12},w_9).
\end{align*} 
 
\begin{figure}[ht]
\includegraphics[width=0.49\textwidth]{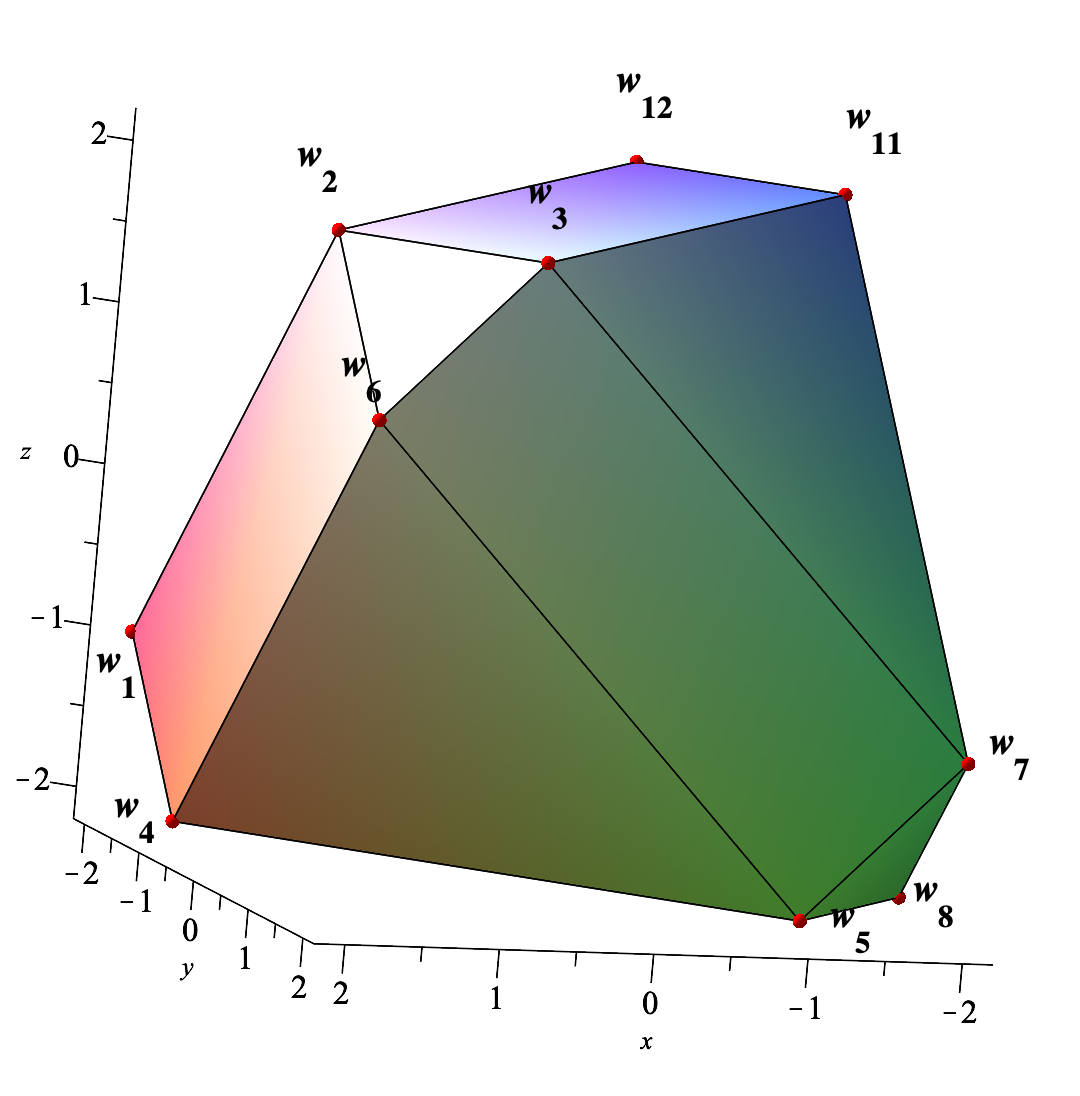}
\includegraphics[width=0.49\textwidth]{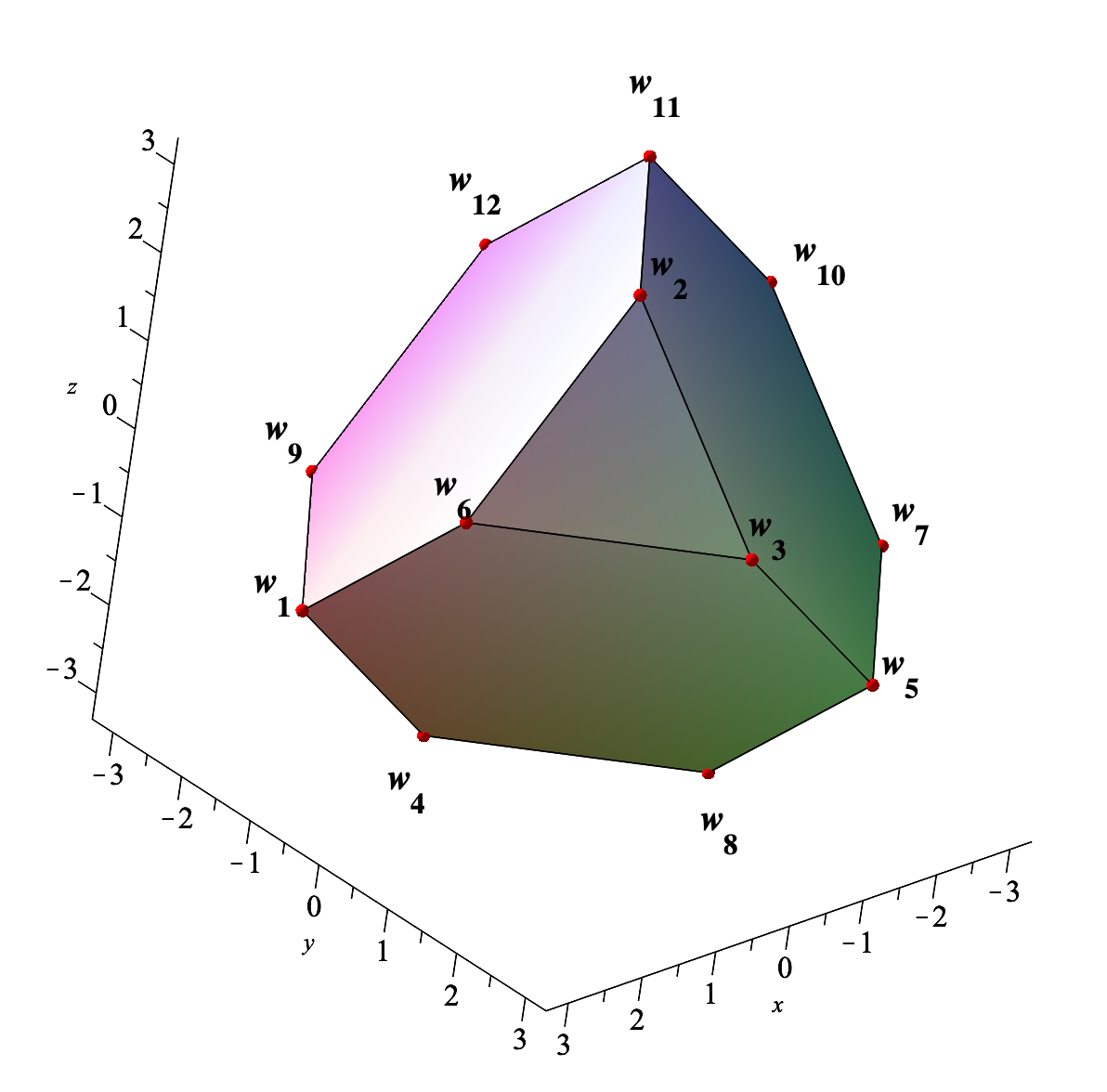}
\caption{The polytope $Q_w$ with $w=(2,2,1)$ (left) and with 
$w=(3,1,1)$ (right). \label{fig:trunctetraandcuboct}}
\end{figure}

\FloatBarrier

\section{Discussion and Open Questions}
\label{subsec:discusquestions}

In this section, we formulate some question which may provide interesting directions of further research.

Firstly, recall that Izmestiev's Colin de Verdi\`ere Matrix formula in Theorem \ref{thm:ivan} applies to general $k$-polytopes in $\R^k$. In the case of simple polytopes, our volume ratio formula in Theorem \ref{thm:main1} provides an alternative expression of the off-diagonal entries of this Colin de Verdi\`ere matrix in the case of simple polytopes (Corollary \ref{cor:altIavnCvD}). It would be interesting to investigate possible generalizations of our volume ratio formula beyond simple cones, which would lead to a corresponding generalization of Corollary \ref{cor:altIavnCvD}.

Regarding vertex transitive polytopes, we note that we only provided proofs of our statements for the example in Subsection \ref{subsec:coxperm} and relied on Maple computations in Subsection \ref{subsec:rottet}. Formal proofs of the statements in Subsection \ref{subsec:rottet} would most likely apply only to this specific example and be lengthy. We consider these statements as motivating observations of more general phenomena, which one should attempt to formulate rigorously and prove in its full generality.

For any finite group $\Gamma \subset O(k)$, the variation obtained by the map $$S^{k-1}(r) \ni w \mapsto Q_w = \conv(\Gamma w)$$ partitions $S^{k-1}(r)$ into subsets of polytopes of the same combinatorial type (distinguished via their face lattices, see \cite[Section 3.2]{G-03} or \cite[Section 2.2]{Z-95}). This leads to the following natural questions:
\begin{itemize}
\item[(a)] Are the connected components of these subsets of $S^{k-1}(r)$ \emph{submanifolds} of $S^{k-1}(r)$, providing a stratification of the sphere? If true, we refer to these submanifolds as the strata of $S^{k-1}(r)$.
\item[(b)] If (a) is true, does for each stratum $\mathcal{S}$ the following hold: $\lambda_1(w):=\lambda_1(T(Q_w))$ is a continuous function on $\mathcal{S}$, which assumes a \emph{global minimum inside the stratum}?
\item[(c)] If (a) and (b) are true, are all polytopes $Q_w$ corresponding to global minima of $\lambda_1(w)$ in any stratum $\mathcal{S}$ \emph{equilateral}?
\end{itemize}

Finally, let us discuss a possible generalization of Godsil's concept of eigenpolytopes (see \cite{G-78}). For a finite simple graph $G=(V,E)$ with vertex set $V = \{v_1,\dots,v_n\}$, adjacency matrix $A_G$, and fixed  eigenvalue $\theta$ of multiplicity $k$, we can construct a convex polytope in $\R^k$ as follows: Choose an arbitrary eigenbasis $\phi_1, \dots, \phi_k$ of $A_G$ corresponding to the eigenvalue $\theta$ and derive $w_1,\dots,w_n \in \R^k$ via the identity
\begin{equation} \label{eq:wphi-rel}
\begin{psmallmatrix} \zsvek{w_1} \\[-1mm] \vdots \\ \zsvek{w_n} \end{psmallmatrix} = \begin{psmallmatrix} \ssvek{\phi_1} & \cdots & \ssvek{\phi_k} \end{psmallmatrix}. 
\end{equation}
Then the associated \emph{$\theta$-eigenpolytope} is given by
$$ P_G(\theta) := \conv\set{w_i}{i \in [n]} \subset \R^k. $$
Moreover, there is a canonical map 
\begin{equation}
\label{eq:vertex-identification}
V \to \{v_1,\dots,w_n \}, \quad v_i \mapsto w_i. 
\end{equation}
Note that the combinatorial type of the $\theta$-eigenpolytope does not depend on the choice of basis, since any two $\theta$-eigenpolytopes are related via an invertible linear transformation of $\R^k$ (see, e.g., \cite[Theorem 2.1]{W-20}).
Conversely, given a convex polytope $P \subset \R^k$ with $n$ vertices $v_1,\dots, v_n \in \R^k$ and its associated vertex-edge graph $G=(V,E)$,  whose adjacency matrix $A_G$ has an eigenvalue $\theta$ of multiplicity $k$, following Winter in \cite[p. 5]{W-20}, we say that $P$ is \emph{$\theta$-spectral}, if the vertex-edge graph of the associated $\theta$-eigenpolytope $P_G(\theta)$ agrees with $G$ under the canonical map \eqref{eq:vertex-identification}.
It was mentioned in \cite[Observation 2.5(i)]{W-20} that there are no known $\theta$-spectral polytopes for $A_G$-eigenvalues $\theta$ which are not the second largest eigenvalue of $A_G$.

\begin{figure}[ht]
\includegraphics[width=0.44\textwidth]{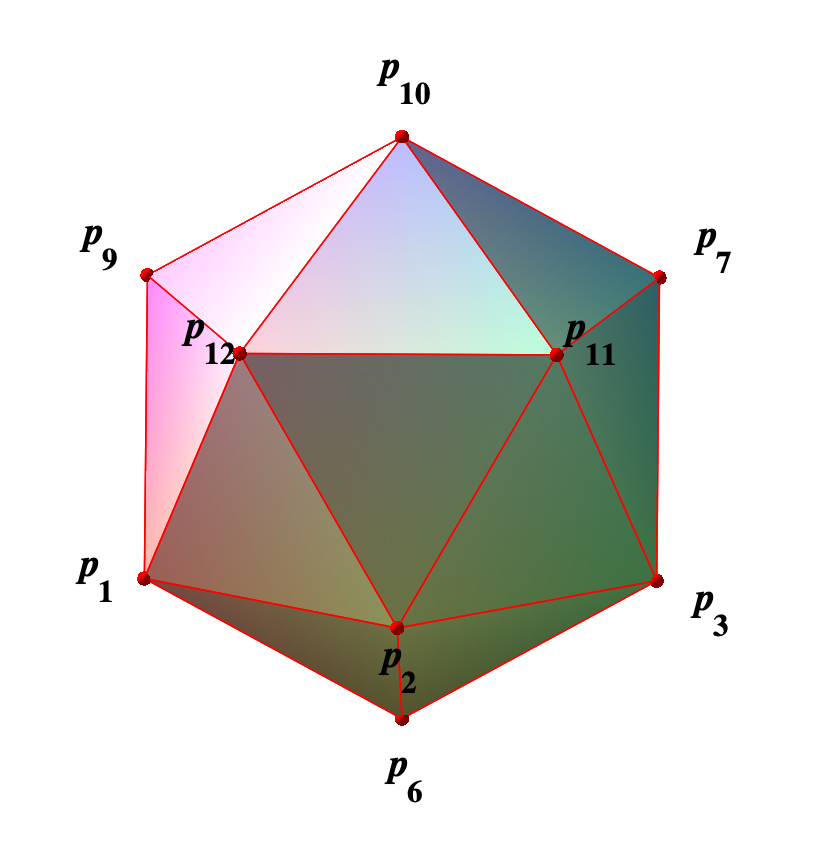}
\includegraphics[width=0.54\textwidth]{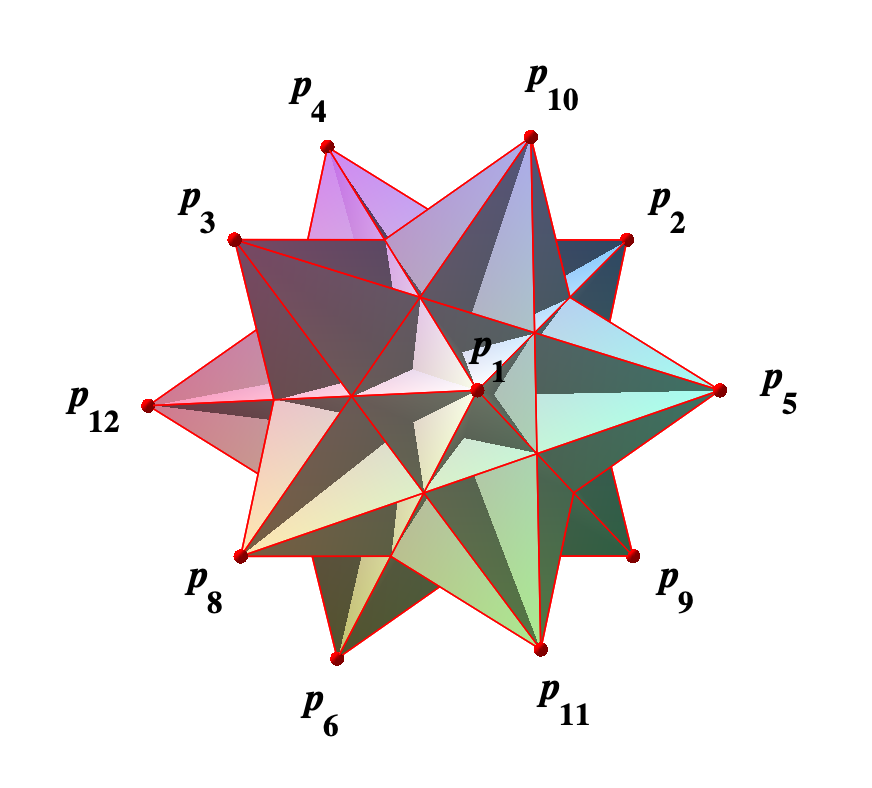}
\caption{Polytopes constructed via orthonormal eigenbases to the eigenvalues $\lambda_1 = \frac{1}{\sqrt{5}}$ (left) and $\lambda_3 = \frac{-1}{\sqrt{5}}$ (right) of $T(Q_w)$ for the choice $(x_1,x_2,x_3)=(\frac{1}{5},\frac{1}{5},\frac{1}{5})$. The polytope on the left hand side is the icosahedron and on the right hand side is Kepler's great icosahedron.
\label{fig:stellation}}
\end{figure}

\FloatBarrier

Instead of using the adjacency matrix, one could also use eigenvalues $\lambda$ and eigenbases of Izmestiev's Colin de Verdi\`ere matrix $M(Q)$ of a convex polytope $Q \subset \R^k$, to define \emph{$\lambda$-spectrality} of this polytope in precisely the same way. It follows from Corollary \ref{cor:specrep} that every convex $k$-polytope $Q \subset \R^k$ is $0$-spectral with respect to the matrix $M(Q)$ in Theorem \ref{thm:ivan}. In the case of vertex transitive polytopes $Q$, the kernel of the Colin de Verdi\`ere matrix $M$ agrees with the $\lambda_1$-eigenspace of the corresponding random walk matrix $T(Q)$. In the example discussed in Subsection \ref{subsec:rottet}, the eigenvalues $\lambda_1(w)$ and $\lambda_3(w)$ of $T(Q_w)$ have multiplicity $3$ for all $w \in \mathcal{M}_{\rm{icosa}}$ (see Figure \ref{fig:eigenvalues} (left)). 
Utilizing eigenbases $\phi_1,\phi_2,\phi_3$ of the lower eigenvalue $\lambda_3(w)$ (instead of $\lambda_1(w)$) for $w \in \mathcal{M}_{\rm{icosa}}$ leads to not necessarily convex polytopes, whose vertices are derived via the identity \eqref{eq:wphi-rel} and whose face-lattices agree with the face-lattices of the original polytopes $Q_w$. Similarly as in the $\lambda_1$-case, it would be interesting to investigate potential relations between extremality of such lower eigenvalues and equilaterality of such non-convex polytopes for arbitrary finite groups $\Gamma \subset O(k)$. The so-constructed polytopes are closely related to the \emph{spectral representations } discussed in \cite[Section 1.2]{IP-13}. In the case $(a,b,c) = (\phi,1,0)$ corresponding to $(x_1,x_2,x_3)=(1/5,1/5,1/5)$,
the matrix $T(Q_w)$ given in \eqref{eq:TQw} has the simple eigenvalue $1$, eigenvalues $\lambda_{1,3} = \pm 1/\sqrt{5}$ (both of multiplicity $3$) and $\lambda_2 = \lambda_4 = -1/5$ of multiplicity $5$. The polytopes obtained via orthonormal bases of the eigenspaces of $T(Q_w)$ to the eigenvalues $\lambda_1$ and $\lambda_3$ are illustrated in Figure \ref{fig:stellation}.

\end{document}